\documentclass{aims}

\usepackage{etex}
\usepackage{amsfonts,amssymb}
\usepackage{amsmath}
\usepackage{cite}
\usepackage{paralist}
\usepackage{array}
\usepackage{ragged2e}
\usepackage[english]{babel} 
\usepackage{ifthen}
\usepackage{bbm}		
\usepackage{graphicx}
\usepackage{epstopdf}
\usepackage{lmodern}
\usepackage{algorithmic}
\usepackage{float}
\usepackage{wrapfig} 
\usepackage{graphicx}
\usepackage{cancel}
\usepackage{mathrsfs}
\usepackage{upgreek}
\usepackage{tikz}	

\usepackage{lineno}

\usetikzlibrary{arrows,shapes,decorations.pathreplacing,decorations.markings,patterns,calligraphy}
\usepackage{pgfplots}

\usepackage{mathtools} 
\usepackage{ragged2e}
\usepackage{tabularx}
\usepackage{nicefrac}
\usepackage{multirow}
\usepackage{enumerate}
\usepackage{yhmath}
\usepackage{bm,upgreek}
\usepackage[colorlinks=true]{hyperref}
\hypersetup{urlcolor=blue, citecolor=red}

\def \d{\, \textup{d}}
\def \T{\textup{T}}
\def \u{\textup{\textbf{u}}}
\def \caSD{\textup{\textbf{R}}}
\def \Rbar{\widebar{\textup{\textbf{R}}}}

\def \diag{\textup{diag}}
\def \R{R}
\def \Riemann{\mathcal{R}}
\def \D{\textup{D}}
\def \poly{\mathcal{P}}

\def \IR{\mathcal{I}_{R}}

\def \M{\mathcal{M}}
\def \Q{\mathcal{Q}}

\def \Mequivalent{\widetilde{\mathcal{M}}}
\def \Mtilde{\mathcal{M}}

\def \H{\mathcal{H}}
\def \L{\mathcal{L}}

\def \smallestEV{\sigma_{\min}}
\def \largestEV{\sigma_{\max}}
\def \ST{\theta}

\def \NumFlux{\boldsymbol{\bar{F}}}

\def \mub{\hat{\mu}}
\def \mut{\hat{\mu}^*}
\def \mud{\hat{\mu}^s}

\def \Rayleigh{\mathcal{Q}}
\def \RayleighQ{\mathcal{Q}^{s}}

\def \Hlinear{\textbf{H}}
\def \Mlinear{\textbf{M}}
\def \intL{\int_0^L}

\DeclareFontFamily{U}{mathx}{\hyphenchar\font45}
\DeclareFontShape{U}{mathx}{m}{n}{<-> mathx10}{}
\DeclareSymbolFont{mathx}{U}{mathx}{m}{n}
\DeclareMathAccent{\widebar}{0}{mathx}{"73}

\def\currentvolume{ }
 
  \def\currentyear{ }
   \def\currentmonth{ }
    \def\ppages{ }

\newtheorem{proposition}{Proposition}

\theoremstyle{definition}

\newtheorem{remark}{Remark}

\definecolor{gruen}{rgb}{0,0.5,0}
\definecolor{braun}{rgb}{0.5,0.25,0.25}
\usetikzlibrary{arrows,shapes,decorations.pathreplacing,decorations.markings}
\xdefinecolor{rwthblue}{rgb}{0,0.3294,0.6235}
\xdefinecolor{rwthlightblue}{rgb}{0.5569,0.7294,0.8980}

\title{Numerical boundary control for semilinear hyperbolic systems} 

\author[Stephan Gerster, Felix Nagel, Aleksey Sikstel, Giuseppe Visconti]{}

	

 \email{stephan.gerster@gmail.com}
 \email{felix.nagel@rwth-aachen.de}
 \email{sikstel@mathematik.tu-darmstadt.de}
 \email{giuseppe.visconti@uniroma1.it}

\begin{document}
\maketitle

\centerline{\scshape Stephan Gerster}
\medskip
{\footnotesize
 \centerline{Universit\`{a} degli Studi dell'Insubria, Como, Italy}}

\medskip

\centerline{\scshape Felix Nagel}
\medskip
{\footnotesize
	\centerline{RWTH Aachen University,  Germany
	}}

\medskip

\centerline{\scshape Aleksey Sikstel}
\medskip
{\footnotesize
	\centerline{Technische Universität Darmstadt, Germany            
	}}

\medskip

\centerline{\scshape Giuseppe Visconti}
\medskip
{\footnotesize
	\centerline{Sapienza Universit\`{a} di Roma, Italy
}}

\bigskip


\begin{abstract}

This work is devoted to the  design of boundary controls of physical systems that are described by semilinear hyperbolic balance laws. 
A computational framework is presented that yields sufficient conditions for a boundary control to steer the system towards a desired state. The presented approach is based on a Lyapunov stability analysis and a CWENO-type reconstruction.


\end{abstract}

\section*{Introduction}

Systems of hyperbolic partial differential equations model fluid flow, chemotaxis  and viscoplastic material dynamics~\cite{O1,Leveque,Kac1974,Inelasticity}. 
Boundary stabilization of these problems has been  studied intensively in the past years~\cite{O1,O2}. 
An underlying tool for the study of these problems are Lyapunov functions that yield upper bounds on the deviation from steady states in suitable norms. 
The virtue of this approach are control rules that  do not require the solution of the whole system, but  take  only measurements at the boundaries into account. 
So-called~\emph{dissipative} boundary conditions~\cite{L1,L2,L4,L5} 
ensure exponential decay of a continuous Lyapunov function, which in turn guarantees that the solution  converges exponentially fast to a desired steady state.  

More precisely, a general theory for the stabilization of linear conservation laws  with respect to the~$L^2$-norm  is available~\cite[Sec.~3]{O1}. 
For nonlinear systems, however, results are still partial. A problem is posed by the fact that Lyapunov's indirect method~\cite{Khalil} does not necessarily hold for hyperbolic systems. Furthermore,  solutions to systems of conservation laws exist in the classical sense only for a finite time due to formation of shocks~\cite{Riemann1859}. 
To this end, stability results are typically stated in terms of the Sobolev ${H^2}$-norm or in the $C^1$-norm~\cite{L5,L4,Hayat2019,Hu2019} and 
restrictive smoothness assumptions on the ${H^2}$-norm of the initial data may be needed~\cite[Sec.~4]{O1}. 
Furthermore, 
most analytical results are based on the assumption that the influence of the source term is  small or in intuitive terms, the
considered balance laws are viewed as perturbations of conservation laws~\cite{L1}. 
On the other hand, if the destabilizing effect of the source term is sufficiently large, the system may be not stabilizable~\cite{Bastin2011,Gugat2019,GugatHerty2020}.

Recently, interest has increased in studying the stabilizability of semilinear hyperbolic systems, when the advection part is linear, but the source term is nonlinear~\cite[Sec.~10]{Petit}.  
In particular, analytical results are available for semilinear Euler equations~\cite{GugatHerty2020,Gugat2021,GugatSemiLinear2022}. 
Assumptions on initial data are typically imposed with respect to the ${H^1}$-Sobolev norm and, hence, are  less restrictive than for  general nonlinear systems. 
The assumption of a Lipschitz continuous source term even allows to establish estimates in terms of the $L^2$-norm~\cite{Prieur2018,Hayat2021}, which is desirable, since also discontinuous initial and boundary data can be treated. Analytical results for general semilinear systems, however, will always come along with restrictions on initial data. In particular, initial data must be sufficiently close to a steady state, where the distance is measured in a suitable norm, for instance in terms of the~$L^2$-~or~$H^1$-norm in space. Otherwise, a blow up of the solution in finite time may occur~\cite[Sec.~2]{Prieur2018}. 

These restrictions motivate a computational approach. 
Indeed, the development and analysis of numerical schemes that preserve continuous
stability results is an active field of research. In particular,  first-order upwind  discretizations of linear balance laws are used to construct discretized Lyapunov functions that decay exponentially  fast~\cite{D2,Gerster2019,Weldegiyorgis2020,Weldegiyorgis2021}. Moreover, a second-order scheme applied to scalar nonlinear conservation laws with dissipative feedback boundary conditions is analyzed in~\cite{Dus2022}. \\

\noindent
The main contribution of this work is a computational framework that is specifically taylored to semilinear boundary value problems 
without any smoothness assumptions on initial data. 
In contrast to most analytical results in the~$H^1$-setting, we measure the distance to a desired state by the  $L^2$-norm and use a Lyapunov function as upper bound, which allows to consider discontinuous solutions.  
Since no systematic procedure exists for deriving Lyapunov functions in this setting, we use a candidate Lyapunov function that is typically applied in the  linear case~\cite[Sec.~3]{O1}.  We parameterize it  up to a constant~$\mub>0$ that is computed numerically. 
Two approaches are presented and analyzed for determining an appropriate constant. Those are based on a weighted Rayleigh coefficient and eigenvalue estimates. 
The computational framework  is consistent with existing theoretical results, but also allows to investigate numerically problems which are beyond the current state of research on analytical control rules. 
The proposed method is based on high-order CWENO reconstructions~\cite{Semplice2019,Visconti2018,Cravero2018,Elena} that are  high-order accurate in smooth regions, but can  resolve discontinuities in an essentially nonoscillatory~(ENO) fashion. 
CWENO reconstructions  consist of weighted combinations of local reconstructions on different stencils. 
Furthermore, they also allow reconstructing  source terms and  boundary conditions at high order, a crucial feature when solving semilinear boundary value problems for balance laws. \\



\noindent
This paper is structured as follows. Section~\ref{SecSemiLinear} reviews semilinear hyperbolic boundary value problems. Section~\ref{SectionConditions} is devoted to their control. 
In particular, Lyapunov functions and their estimates as well as benchmark problems are introduced. 
Section~\ref{ComputationalFramework} describes the computational framework which is based on high-order  CWENO discretizations.
Finally, numerical results are presented in Section~\ref{SecNumericalResults}.

\section{Semilinear hyperbolic boundary value problems}\label{SecSemiLinear}

We consider semilinear hyperbolic balance laws of the form
\begin{equation}\label{SemiLinear}
\partial_t \u(t,x)
+
A(x) \partial_x \u(t,x)
=-
h\big(\u(t,x);x\big)
\end{equation}
that  are defined on a finite space interval~$[0,L]$. 
We consider~$2 \times 2$ systems, where the advection part~$A(x)=T(x) \Lambda(x) T^{-1}(x) $ is diagonalizable with distinct eigenvalues 
$$
\Lambda^-(x)
<0<
\Lambda^+(x)
\quad
\text{for all}
\quad
x\in[0,L]
\quad
\text{and}
\quad
\Lambda(x)
=
\diag\big\{
\Lambda^+(x),
\Lambda^-(x)
\big\}.
$$ 
Under the assumptions 
$
A(x)\in C^2\big(
[0,L];\mathbb{R}^{2\times 2}
\big)
$ 
and 
$
h\in C^2\big(
\mathbb{R}^2 \times [0,L];\mathbb{R}^2
\big)
$ 
the semilinear system~\eqref{SemiLinear} admits a classical smooth solution provided that initial data are differentiable~\cite{BRESSAN,O1}. Hence, it can be equivalently written in \textbf{Riemann invariants} 
$
\R(t,x)
\coloneqq
T^{-1}(x)
\u(t,x)
$ 
satisfying
\begin{equation}\label{DiagonalizedSystem}
\begin{aligned}
&\partial_t \R(t,x)
+
\Lambda(x)
\partial_x \R(t,x)
=-
G\big( \R(t,x);x \big) \\
&\text{for}
\quad
G\big( \R(t,x);x \big)
=
T^{-1}(x)
h\big(
T(x)\R(t,x);x
\big)
+
T^{-1}(x)\Lambda(x) \partial_x T(x)  \R(t,x).
\end{aligned}
\end{equation}
The diagonalized system~\eqref{DiagonalizedSystem} is endowed with possibly nonlinear feedback boundary conditions~$
B\in C^2\big(
\mathbb{R}^2;\mathbb{R}^2
\big)
$. 
The \textbf{initial boundary value problem (IBVP)} with  initial values $\IR\in L^2\big( (0,L); \mathbb{R}^2 \big)$, which satisfy boundary conditions,  reads as 
\begin{align}
\partial_t \R(t,x) + \Lambda(x) \partial_x \R(t,x) &=  -G\big( \R(t,x);x \big)
& & \text{for} \ \ t \in (0,T), \ x \in (0,L),  \label{CauchyProblem1} \\
\begin{pmatrix} \R^+(t,0) \\ \R^-(t,L) \end{pmatrix} 
&=
{B}
\begin{pmatrix} \R^+(t,L) \\ \R^-(t,0) \end{pmatrix} 
& & \text{for} \ \ t \in [0,T),  \label{CauchyProblem2} \\
\R(0,x)&= \IR(x)
& & \text{for} \ \ x \in [0,L].
\label{CauchyProblem3} 
\end{align}
Riemann invariants that come along with positive  speeds~$\lambda^+(x)>0$ are denoted as~$\R^+(t,x)$ and those with negative characteristic speeds as~$\R^-(t,x)$, respectively. 
Typical examples, see~e.g.~\cite[Sec.~1.11]{O1}, are the 
\textbf{Kac-Goldstein equations}, which explain the spatial pattern formations in chemosensitive populations.  The unknowns~$\u=(\rho,q)^\T$ are the  density~$\rho=\R^++\R^-$ and the mass flux~$q=\gamma(\R^+-\R^-)$ 
of right~{($\R^+$)} and left-moving~{($\R^-$)} cells, 
where  the velocity of cell motion is described by the parameter~$\gamma>0$ 
and 
$
\ST(\R^+,\R^-)
$ 
is a turning function.
\begin{align*}
&\textbf{Kac-Goldstein equations} 
&
&\ \ \textbf{diagonalized form} \\
&
\begin{aligned}
\rho_t
+
q_x
&=
0, \\
q_t 
+
\gamma^2
\rho_x
&=
2\ST\bigg(
\frac{\rho}{2}+\frac{q}{2\gamma},
\frac{\rho}{2}-\frac{q}{2\gamma}
\bigg)q
\end{aligned}
& & \ \
\begin{aligned}
\R^+_t
+
\gamma \R^+_x
&=
\ST(\R^+,\R^-) (\R^+ - \R^-), \\
\R^-_t 
-
\gamma
\R^-_x
&=
\ST(\R^+,\R^-) (\R^- - \R^+)
\end{aligned}
\end{align*}


\noindent
\textbf{Steady states} are denoted by  $\bar{\u}=(\bar{\rho},\bar{q})^\T$, ~$\bar{\R}=(\bar{\R}^+,\bar{\R}^-)^\T$. Those satisfy the conditions~$\partial_t \bar{\u}(x)=\partial_t \u(t,x)=0$, 
$\partial_t \bar{\R}(x)=\partial_t \R(t,x)=0$
 and are typically space-dependent. In the special case of Kac-Goldstein equations, however, the steady states are constant and read as
$$
\bar{\R}^\pm
=
\frac{1}{2}
\int_0^L 
\rho(0,x) \d x.
$$
\textbf{Boundary conditions} are specified by
\begin{equation} \label{BCKAC}
\begin{pmatrix} \R^+(t,0) \\ \R^-(t,L) \end{pmatrix} 
=
\begin{pmatrix}
& \hspace{-2mm}\kappa \\ \kappa 
\end{pmatrix}
\begin{pmatrix} \R^+(t,L) \\ \R^-(t,0) \end{pmatrix} 
\qquad
\text{for}
\qquad
|\kappa| \in [0,1].
\end{equation}
In particular the choice~$\kappa=1$ models the case, when cells are confined within  the spatial domain, 
since it holds~$q(t,0)=0$ and~$q(t,L)=0$.  


\noindent
According to~\cite[Th.~10.1]{Petit},  there exists  the following wellposedness result.  
Provided that initial data are sufficiently smooth and close to a steady state, i.e.
\begin{equation}\label{IVsmooth}
\text{there exists}
\quad
\delta>0
\quad
\text{such that}
\quad
\big\lVert
\u_0 - \bar{\u}
\big\rVert_{H^1\big((0,L);\mathbb{R}^{2}  \big)}
<\delta,
\end{equation}
the IBVP~\eqref{CauchyProblem1} -- \eqref{CauchyProblem3} 
has a unique maximal classical solution satisfying
$$
\u 
\in 
C^0\Big(
[0,T);
H^1\big( (0,L);\mathbb{R}^{2} \big)
\Big).
$$
Under the assumption~\eqref{IVsmooth}, there are 
conditions  available,  see~e.g.~\cite[Th.~10.2]{Petit}, that stabilize the dynamics at a steady state. 
This assumption, however, is relatively restrictive for hyperbolic systems, which may involve discontinuous solutions, e.g.~in the case of time-dependent boundary controls and for initial perturbations that are  away from a steady state. 
To this end, we consider in Section~\ref{SectionConditions} stabilization concepts with respect to the $L^2$-norm, which are typically applied to linear systems~\cite[Sec.~5]{O1}, and use them to establish a computational framework for semilinear problems~in~Section~\ref{ComputationalFramework}.

\section{Sufficient conditions for stability}\label{SectionConditions}

We are interested in  boundary controls that make the system converge exponentially fast to a steady state. As  in~\cite[Ch.~10]{Petit}, we introduce the distance to this steady state by~$
\Riemann(t,x) \coloneqq
\R(t,x)-\bar{\R}(x)
$. Then, the IBVP~\eqref{CauchyProblem1} -- \eqref{CauchyProblem3} reads as
\begin{align}
	\partial_t \Riemann(t,x) + \Lambda(x) \partial_x \Riemann(t,x) &=-  \mathcal{G}\big( \Riemann(t,x);x \big)
	& & \text{for} \ \ t \in (0,T), \ x \in (0,L),  \label{IBVP1} \\
	\begin{pmatrix} \Riemann^+(t,0) \\ \Riemann^-(t,L) \end{pmatrix} 
	&=
	\mathcal{B}
	\begin{pmatrix} \Riemann^+(t,L) \\ \Riemann^-(t,0) \end{pmatrix} 
	& & \text{for} \ \ t \in [0,T),  \label{IBVP2} \\
	\Riemann(0,x)&= \IR(x)-\bar{\R}(x)
	& & \text{for} \ \ x \in [0,L]. \label{IBVP3}
\end{align}
The source term and the boundary conditions are defined by
\begin{align*}
	\mathcal{G}\big( \Riemann;x \big)
	&\coloneqq
	G\big( \Riemann + \bar{\R}(x) ;x \big)
	-
	G\big( \bar{\R}(x) ;x \big), \label{SourceTerm}  \\
	\mathcal{B}
	\begin{pmatrix} \Riemann^+(t,L) \\ \Riemann^-(t,0) \end{pmatrix} 
	&\coloneqq
	{B}
	\begin{pmatrix} \Riemann^+(t,L)+\bar{R}^+(L) \\ \Riemann^-(t,0)+\bar{R}^-(0) \end{pmatrix} 
	-
	\begin{pmatrix}
		\bar{R}^+(0) \\
		\bar{R}^-(L)
	\end{pmatrix}. 
\end{align*}

\noindent
To specify boundary conditions, we
introduce a Lyapunov function candidate 
\begin{equation}\label{analyticalLF}
	\L(t;\mub) \coloneqq \intL \Riemann(t,x)^{\T} W(x;\mub) \Riemann(t,x) \d x
\end{equation}
with the weights~$W(x;\mub)\coloneqq \diag \big\{ W^+(x;\mub),W^-(x;\mub)  \big\}$ that are defined by
\begin{equation}\label{Weights}
	\begin{aligned}
		W^+(x;\mub)
		&\coloneqq
		\frac{ 1 }{\Lambda^+(x)}
		\textup{exp} \bigg(
		-\mub \int_0^x 
		\frac{1}{\Lambda^+(s)} \d s
		\bigg), \\
		W^-(x;\mub)
		&\coloneqq
		\frac{ 1 }{|\Lambda^-(x)|}
		\textup{exp} \bigg(
		\mub \int_x^L 
		\frac{1}{\Lambda^-(s)} \d s
		\bigg).
	\end{aligned}
\end{equation}


\begin{remark}
		As described in~\cite[Sec.~3.5]{O1} and as illustrated in Figure~\ref{FigureNetwork},  the $2\times 2$~system~{\eqref{CauchyProblem1}~--~\eqref{CauchyProblem3}} is extendable to a network with   $n$~arcs  by specifying appropriate coupling conditions. 
		Then, the weights~\eqref{Weights} must be replaced by those in~\cite[Th.~3.16]{O1}.
		In the sequel, this article is concerned with~$2 \times 2$ systems. 
	\begin{figure}[H]

\tikzstyle{block} = [thick,draw, fill=black!10, rectangle, 
minimum height=3em, minimum width=3em]
\tikzstyle{Kreis} = [thick,draw, fill=blue!20, circle, 
minimum height=3em, minimum width=3em]
\tikzstyle{sum} = [draw, fill=darkgreen!20, rectangle, 
minimum height=2em, minimum width=3em]
\tikzstyle{input} = [coordinate]
\tikzstyle{output} = [coordinate]
\tikzstyle{pinstyle} = [ pin edge={to-,thick,black}]
\tikzstyle{vecArrow} = [thick,double distance=23pt]

\hspace*{-8mm}
\begin{tikzpicture}[scale=0.082]

	\draw [vecArrow] (5,0) -- (50,0);
	\draw [vecArrow] (145,15) -- (80,2.6);
	\draw [vecArrow] (145,-15) -- (80,-2.6);
	
	\node [block,align=center] at (70,0) {\textbf{coupling conditions}\\
		$
		\begin{pmatrix}\color{blue} \R^+(t,0) \\ \color{red} \R^-(t,L) \end{pmatrix} 
		=
		B
		\begin{pmatrix} \color{blue} \R^+(t,L) \\ \color{red} \R^-(t,0) \end{pmatrix} 
		$
	};
	
	\node at (25,0) {arc 1};
	\node at (120,10.5) {arc 2};
	\node at (120,-10.5) {arc $n$};
	
	\node at (15.9,9) {\color{blue}$\R^{+,(1)}(t,x)$};
	\node at (28,-9) {\color{red}$\R^{-,(1)}(t,x)$};
	
	\node at (118,-21) {\color{blue}$\R^{+,(n)}(t,x)$};
	\node at (143,3.5) {\color{red}$\R^{-,(2)}(t,x)$};
	
	\node at (118,21) {\color{blue}$\R^{+,(2)}(t,x)$};
	\node at (143,-3.5) {\color{red}$\R^{-,(n)}(t,x)$};
	
	\node at (68,16) {\color{blue}$
		{
			\R^+= \Big(\R^{+,(1)},\ldots,\R^{+,(n)} \Big)^\T
		}$};
	
	\node at (68,-16) {\color{red}$
		{
			\R^-= \Big(\R^{-,(1)},\ldots,\R^{-,(n)} \Big)^\T
		}
		$};
	
	\draw [ultra thick, blue, decoration={markings,mark=at position 1 with	{\arrow[scale=2,>=stealth]{>}}},postaction={decorate}] (	28,9) -- (38,9);
	\draw [ultra thick, red, decoration={markings,mark=at position 1 with	{\arrow[scale=2,>=stealth]{>}}},postaction={decorate}] (15,-9) -- (5,-9);
	
	\draw [ultra thick, blue, decoration={markings,mark=at position 1 with	{\arrow[scale=2,>=stealth]{>}}},postaction={decorate}]  (132,-21.4) -- (145,-23.5);
	
	\draw [ultra thick, red, decoration={markings,mark=at position 1 with	{\arrow[scale=2,>=stealth]{>}}},postaction={decorate}] (130,-3.8) -- (120,-2);
	
	\draw [ultra thick, red, decoration={markings,mark=at position 1 with	{\arrow[scale=2,>=stealth]{>}}},postaction={decorate}] (130,3.8) -- (120,2);
	
	\draw [ultra thick, blue, decoration={markings,mark=at position 1 with	{\arrow[scale=2,>=stealth]{>}}},postaction={decorate}]  (132,21.4) -- (145,23.5);
	
\end{tikzpicture}

		\caption{Network with~$n$ arcs and coupling conditions~$B:\mathbb{R}^{2n}\rightarrow \mathbb{R}^{2n}$.}
		\label{FigureNetwork}
	\end{figure}
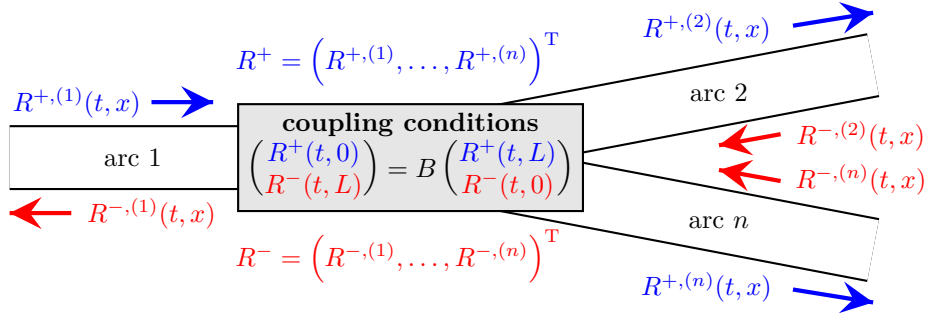
\end{remark}


\noindent
We introduce the notations~$
W^\pm_b\coloneqq W^\pm(b,\mub)$, 
$
\Lambda^\pm_b\coloneqq \Lambda^\pm(b)$, 
	$
	\Lambda_b  \coloneqq
	\diag\big\{
	\Lambda_b^+, \Lambda_b^-
	\big\}
	$,  
	$
	W_b  \coloneqq
	\diag\big\{
	W_b^+, W_b^-
	\big\}
	$, 
$\Riemann^\pm_b(t)\coloneqq \Riemann^\pm(t,b)$ with ${b\in\{0,L\}}$ 
and, for now, we assume 
\begin{equation} \label{ToStrongRegularity}
	\Riemann \in C^1\Big([0,\infty) \times [0,L];\mathbb{R}^{2}\Big).
\end{equation}
According to~\cite[Sec.~10.2]{Petit}, 
there exists a matrix~$\bar{\mathcal{G}}$ satisfying~$
\mathcal{G}(\Riemann;x)
=
\bar{\mathcal{G}}(\Riemann;x)
\Riemann
$ and 
$
\bar{\mathcal{G}}(0;x)
=
\D_{\Riemann} \mathcal{G}(\Riemann;x) \big|_{\Riemann=0}
$. 
This allows to define the matrices
\begin{align*}
	\Mtilde(\Riemann;x,\mub)
	&\coloneqq
	\mub\,
	W(x;\mub) 
	+
	W(x;\mub)\bar{\mathcal{G}}(\Riemann;x)
	+
	\bar{\mathcal{G}}(\Riemann;x)^\T W(x;\mub),\\
	\Mequivalent(\Riemann;x,\mub)
	& \coloneqq
	W^{-\nicefrac{1}{2}}(x,\mub) 
	\Mtilde(\Riemann;x,\mub) 
	W^{-\nicefrac{1}{2}}(x,\mub) , \\
	\H\big(\Riemann^+_L,\Riemann^-_0;\mub\big)
	&\coloneqq
	\mathcal{B}
	\begin{pmatrix} \Riemann^+_L \\ \Riemann^-_0 \end{pmatrix}^{\T}
	\begin{pmatrix}
		W^+_0(\mub)\Lambda^+_0 \\ &\hspace{-6mm}  W^-_L(\mub) |\Lambda^-_L|
	\end{pmatrix}	
	\mathcal{B}
	\begin{pmatrix} \Riemann^+_L \\ \Riemann^-_0 \end{pmatrix} \\
	&\quad \, - 
	\begin{pmatrix} \Riemann^+_L \\ \Riemann^-_0 \end{pmatrix}^{\T}
	\begin{pmatrix}
		W^+_L(\mub)\Lambda^+_L \\ &\hspace{-6mm}  W^-_0(\mub)|\Lambda^-_0|
	\end{pmatrix}	
	\begin{pmatrix} \Riemann^+_L \\ \Riemann^-_0 \end{pmatrix}.
\end{align*}
Then, the time derivative of the Lyapunov function is
\begin{align}
	\L'(t;\mub)
	=& \hspace{4.6mm}
	2 \intL \Riemann(t,x)^{\T}W(x;\mub) \partial_t\Riemann(t,x) \d x \nonumber \\
	=& 
	-2 \intL \Riemann(t,x)^{\T}W(x;\mub)\Lambda(x) \partial_x\Riemann(t,x) \,
	+  \Riemann(t,x)^{\T}W(x;\mub)\mathcal{G}(\Riemann;x) \d x \nonumber \\
	=&
	\intL
	\Riemann(t,x)^{\T} \Big[
	\partial_x \big( W(x;\mub)\Lambda(x) \big)
	-
	2\, W(x;\mub)\bar{\mathcal{G}}(\Riemann;x) 
	\Big]
	\Riemann(t,x) 
	\d x \nonumber \\
	&- \Big[
	\Riemann(t,L)^{\T} W_L(\mub)\Lambda_L \Riemann(t,L)
	-
	\Riemann(t,0)^{\T} W_0(\mub)\Lambda_0 \Riemann(t,0)
	\Big] \nonumber \\ 
	=&
	-\intL\Riemann(t,x)^\T \Mtilde\big(\Riemann(t,x);x,\mub\big)\Riemann(t,x) \d x
	+
	\H\big(\Riemann^+_L,\Riemann^-_0;\mub\big).
	\label{DerivativeLyapunov}
\end{align}


\noindent
We observe from~equation~\eqref{DerivativeLyapunov} that a \emph{sufficient, but not necessary} condition to make the Lyapunov function decay exponentially fast is to choose boundary conditions and a parameter~$\mub\geq 0$ such that
\begin{align}
	&\text{the inequality }
	\H\big(\Riemann^+_L,\Riemann^-_0;\mub\big) \leq 0 
	\text{ holds and }
	\label{A2}\tag{\textup{A1}} \\
	&\text{the matrix }
	\Mtilde(\Riemann;x,\mub)
	\text{ is strictly positive definite for all }
	x\in[0,L]. \label{A1tilde}\tag{$\textup{A2}$} 
\end{align}
Furthermore, we introduce the~\textbf{weighted Rayleigh quotient}
\begin{equation}\label{Rayleigh}
	\begin{aligned}
		\Rayleigh\big[
		\Riemann,\mub
		\big](t)
		&\coloneqq
		\big\lVert \Riemann^{(\mub)} \big\rVert_{L^2}^{-2}
		\left\langle \Riemann^{(\mub)}, \Mequivalent (\Riemann;\cdot,\mub) \Riemann^{(\mub)} \right\rangle_{L^2}
		\quad \ \text{for}\\
		\Riemann^{(\mub)}(t,x)
		&\coloneqq
		W(x;\mub)^{\nicefrac{1}{2}}  \Riemann(t,x)
	\end{aligned}
\end{equation}
where the~$L^2$-norm satisfies~$
\big\lVert \Riemann^{(\mub)} \big\rVert_{L^2}^{2}
\coloneqq
\big\lVert \Riemann^{(\mub)} \big\rVert_{L^2( (0,L);\mathbb{R}^{2} )}^{2}
=
\L(t;\mub)
$. 
Then, 
the time derivative~\eqref{DerivativeLyapunov}  fulfills
\begin{equation*}
	\L'(t;\mub)
	\leq
	-
	\Rayleigh\big[
	\Riemann,\mub
	\big](t)
	\L(t;\mub)
\end{equation*}
provided that~assumption~\eqref{A2} holds.
If the weighted Rayleigh quotient~\eqref{Rayleigh} remains strictly positive, i.e.
\begin{equation} \label{A1}\tag{\textup{A3}} 
	\Rayleigh\big[
	\Riemann,\mub
	\big](t)
	>\mu>0
	\quad
	\text{for all}
	\quad
	t \geq 0,
\end{equation}
the solution converges to the steady state exponentially fast. More precisely, 
the norm equivalence~$
\big\lVert \Riemann(t,x) \big\rVert_2^2
\sim
\big\lVert \Riemann(t,x) \big\rVert_{W(x)}^2
\coloneqq
\Riemann(t,x)^\T W(x) \Riemann(t,x) 
$ and the estimate~$\L'(t)\leq - \mu\L(t) $ imply for a \emph{fixed} parameter~$\mub \geq 0$ the bound
$$
\big\lVert
\Riemann(t,\cdot)
\big\rVert_{
	L^2
}
\lesssim
e^{-\mu t}
\big\lVert
\Riemann(0,\cdot)
\big\rVert_{
	L^2
}
\quad
\text{for all}
\quad
t \in \mathbb{R}^+_0.
$$

\begin{remark}\label{Remark1}
	The conditions~\eqref{A2}~--~\eqref{A1} are coupled by the parameter~$\mub\geq 0$, which enters the weights of the Lyapunov function. More precisely, 
	it has been shown in~\cite[Th.~2.3.5]{StephansDiss} for linear boundary controls, imposed by a matrix~$\bar{\mathcal{B}} \in\mathbb{R}^{2n\times 2n}$,
	that the inequality~\eqref{A2} is satisfied if  the condition
	\begin{equation}\label{StephansDiss}
		\textup{exp} \left(  \frac{\mub L}{2 \lambda_{\min} }  \right)
		\big\lVert \bar{\mathcal{B}} \big\rVert_2 
		\leq
		1
		\quad\text{holds for}\quad
		\lambda_{\min} \coloneqq
		\min_{x \in [0,L]}
		\Big\{ \big| \Lambda^\pm(x) \big| \Big\}.
	\end{equation}
	Hence, a \emph{small} value of~$\mub\geq 0$ is desirable. On the other hand, a \emph{large} value may be necessary to make the conditions~\eqref{A1tilde} and~\eqref{A1} hold. This ambiguity reflects the fact that some systems are not even stabilizable unless the length~$L>0$ is sufficiently small~\cite{Bastin2011,Gugat2019}.
\end{remark}


	\begin{remark}[{Stabilization of~$H^1$-solutions according to~\cite[Th.~10.2]{Petit}}] 
		Provided that initial data are sufficiently close to a steady state, i.e.~condition~\eqref{IVsmooth} is satisfied, 
		exponential stability holds with respect to the $H^1$-norm if the matrix
		$$
		\mub\,
		W(x;\mub) 
		+
		W(x;\mub) \boldsymbol{\bar{\mathcal{G}}}(x)
		+
		\boldsymbol{ \bar{\mathcal{G}}}(x)^\T W(x;\mub)
		$$ 
		is strictly positive definite for all~$x\in[0,L]$ and the matrix
		$$
		\boldsymbol{\bar{\mathcal{B}}}^{\T} 
		\begin{pmatrix}
			W^+_0(\mub) \Lambda^+_0 \\ & W^-_L(\mub) \big| \Lambda^-_L \big|  
		\end{pmatrix}
		\boldsymbol{\bar{\mathcal{B}}}
		-
		\begin{pmatrix}
			W^+_L(\mub) \Lambda^+_L \\ & W^-_L(\mub) \big| \Lambda^-_L \big|  
		\end{pmatrix}
		$$
		is negative semidefinite, 
		where~$
		\boldsymbol{\bar{\mathcal{B}}}
		=
		\D_{\Riemann} \mathcal{B}(\Riemann) \big|_{\Riemann=0}
		$
		and 
		$
		\boldsymbol{\bar{\mathcal{G}}}(x)
		=
		\bar{\mathcal{G}}(0;x)
		$ 
		denote linearizations at steady  state. 
		Hence, assumptions~\eqref{A2} and~\eqref{A1tilde} are \emph{more restrictive} as they must be satisfied also apart from the steady state. 	
		
	\end{remark}

	\noindent
	The stabilization concept with respect to the $L^2$-norm, which we follow in this work, does not serve as a general stability result.  However, it 
	comes along with a computational framework that allows for an efficient numerical verification of the conditions~\eqref{A2}~--~\eqref{A1} and hence  allows to investigate numerically problems where initial data may vary widely from steady states. 
	It is justified analytically  in the special cases considered in Section~\ref{SubSectionI} and  Section~\ref{SecLipschitz}. 


\subsection{Linearized case}\label{SubSectionI}
In the linear case, when the source term and the boundary conditions are given by matrices, i.e.~$\bar{\mathcal{G}}\in\mathbb{R}^{2n\times 2n}$ and 
$
\bar{\mathcal{B}}\in\mathbb{R}^{2n\times 2n}
$, 
the Lyapunov function decays exponentially fast if the the matrix~$\Mlinear(x,\mub)=\Mtilde(\Riemann;x,\mub)$ 
is strictly positive definite for all~$x\in[0,L]$ and the matrix
\begin{equation}\label{HLyapunov}
	\Hlinear(\mub) \coloneqq
	\bar{\mathcal{B}}^{\T} 
	\begin{pmatrix}
		W^+_0(\mub) \Lambda^+_0 \\ & W^-_L(\mub) \big| \Lambda^-_L \big|  
	\end{pmatrix}
	\bar{\mathcal{B}}
	-
	\begin{pmatrix}
		W^+_L(\mub) \Lambda^+_L \\ & W^-_L(\mub) \big| \Lambda^-_L \big|  
	\end{pmatrix}
\end{equation}
is negative semidefinite.  
More precisely, the derivative~\eqref{DerivativeLyapunov} is estimated by
\begin{align*}
	&\L'(t;\mub)
	\leq
	-\intL  \Riemann(t,x)^\T \Mlinear(x;\mub) \Riemann(t,x) \d x
	\leq
	-\mu \L(t;\mub)
	\quad
	\Leftrightarrow
	\quad
	\L(t) \leq e^{-\mu t} \L(0) \\
	&
	\text{with decay rate}\quad
	\mu \coloneqq \min_{x \in [0,L]}\Big\{ \smallestEV \Big\{ W^{-\nicefrac{1}{2}}(x;\mub) \Mlinear(x;\mub)W^{-\nicefrac{1}{2}}(x;\mub)  \Big\} \Big\}>0,
\end{align*}
where~$\smallestEV$ denotes the smallest eigenvalue of a matrix. 
An example, used in the following as a benchmark problem, is as follows:

\begin{proposition}\label{Corollary1}
	The linear boundary value problem
	$$
	\begin{aligned}
		\partial_t
		\begin{pmatrix} \Riemann^+(t,x) \\ \Riemann^-(t,x) \end{pmatrix}
		+
		\partial_x
		\begin{pmatrix} \hspace{2.6mm}\Riemann^+(t,x) \\ -\Riemann^-(t,x) \end{pmatrix}
		&=
		\frac{\ST(x)}{2} \begin{pmatrix}
			\hspace{2.5mm} 1 & \hspace{-2.5mm}-1 \\ -1 & 1
		\end{pmatrix}
		\begin{pmatrix} \Riemann^+(t,x) \\ \Riemann^-(t,x) \end{pmatrix},\\
		\begin{pmatrix} \Riemann^+(t,0) \\ \Riemann^-(t,L) \end{pmatrix}
		&=
		\begin{pmatrix}
			&\hspace{-2mm} \kappa \\ \kappa 
		\end{pmatrix}
		\begin{pmatrix} \Riemann^+(t,L) \\ \Riemann^-(t,0) \end{pmatrix}
	\end{aligned}
	$$
	is exponentially stable provided that the inequality
	\begin{equation}\label{ConditionStable}
		\Big|
		\kappa \ln\big(|\kappa|\big)
		\Big| >{  \alpha L} 
		\quad\text{holds for}\quad
		\alpha \coloneqq \max\limits_{x\in[0,L]}
		\Big\{ \big|
		\ST(x)
		\big|  \Big\}
		\quad\text{and}\quad
		|\kappa|\in(0,1).
	\end{equation}
	Furthermore, there exists a value~$|\kappa|\in(0,1)$ in the case~$\alpha L<\nicefrac{1}{e}$.
\end{proposition}

\begin{proof}
	According to~\cite[Th.~2.3.5]{StephansDiss}, the matrix~\eqref{HLyapunov} is negative semidefinite for
	$$
	1\geq \textup{exp} \Big( \frac{\mub L}{2}  \Big)
	\big\lVert \bar{\mathcal{B}} \big\rVert_2 
	=
	\textup{exp} \Big( \frac{\mub L}{2}  \Big) |\kappa|
	\quad\Leftrightarrow\quad
	\mub \leq \bar{\mu}\coloneqq
	-\frac{2}{L} \ln\big( {|\kappa|} \big),
	\quad
	|\kappa|\in(0,1).
	$$
	\noindent
	The choice~$\mub=\bar{\mu}$ yields the weights
	$$	
	w^+(x;\bar{\mu})
	=
	e^{-\bar{\mu} x}
	=
	\kappa^{\frac{2x}{L}}
	\quad\text{and}\quad
	w^-(x;\bar{\mu})
	=
	e^{-\bar{\mu} (L-x)}
	=
	\kappa^{\frac{2(L-x)}{L}}.
	$$	
	Furthermore, the maximum of the sum~$w^+(x;\bar{\mu})^p+w^-(x;\bar{\mu})^p$, $p\in\mathbb{N}$ is obtained at~$x=\nicefrac{L}{2}$, which yields the upper  bounds~$w^\pm(x;\bar{\mu})\leq 1$,~$
	w^+(x;\bar{\mu})+w^-(x;\bar{\mu})\leq 2|\kappa|
	$ and~$w^+(x;\bar{\mu})^2+w^-(x;\bar{\mu})^2\leq 2\kappa^2
	$. 
	This gives the  eigenvalue estimate
	\begin{align}
		&\begin{aligned}
			\smallestEV\big\{
			\Mlinear(x;\bar{\mu}) 
			\big\}
			&\geq 
			\bar{\mu} W(x;\bar{\mu})
			-
			\largestEV
			\Big\{ 
			W(x;\bar{\mu})\bar{\mathcal{G}}(x)+\bar{\mathcal{G}}(x)^{\T}W(x;\bar{\mu})
			\Big\}
			\nonumber \\
			&\geq 
			\bar{\mu} W(x;\bar{\mu})
			-
			\alpha
			\sqrt{ \frac{ w^+(x;\bar{\mu})^2+w^-(x;\bar{\mu})^2}{2} } 
			+\alpha
			\frac{ w^+(x;\bar{\mu})+w^-(x;\bar{\mu})}{2} 
		\end{aligned}
		\nonumber \\
		&\Longrightarrow \quad
		\min_{x \in [0,L]}\Big\{
		\smallestEV\big\{
		\Mlinear(x;\bar{\mu}) 
		\big\}
		\Big\}
		\geq 
		-\frac{2}{L} \ln\big(|\kappa|\big) \kappa^2
		- 2\alpha |\kappa|, \label{BoundStable}
	\end{align} 
	where~$\largestEV$ denotes the spectral radius. The bound~\eqref{BoundStable} is strictly positive if the condition~\eqref{ConditionStable} holds.


\end{proof}

\subsection{Semilinear case with Lipschitz continuous source term}\label{SecLipschitz}
Similarly to~\cite{Hayat2021}, we consider Lipschitz continuous source terms. 
Then, results of the linearized case can be partially extended 
as shown in the following proposition.

\begin{proposition}\label{Corollary2}
	The semilinear boundary value problem
	$$
	\begin{aligned}
		\partial_t \Riemann(t,x)
		+
		\Lambda(x)
		\partial_x \Riemann(t,x)
		&=-
		\mathcal{G} \big( \Riemann(t,x);x \big) \\
		\begin{pmatrix} \Riemann^+(t,0) \\ \Riemann^-(t,L) \end{pmatrix}
		&=
		\begin{pmatrix}
			&\hspace{-2mm} \kappa \\ \kappa 
		\end{pmatrix}
		\begin{pmatrix} \Riemann^+(t,L) \\ \Riemann^-(t,0) \end{pmatrix}
	\end{aligned}
	$$
	with Lipschitz continuous source term
	$$
	\big\lVert
	\mathcal{G}(\Riemann;x) 
	\big\rVert_2
	\leq C_{\mathcal{G}} \big\lVert
	\Riemann
	\big\rVert_2
	\quad\text{for all}\quad
	x\in[0,L]
	$$
	is exponentially stable provided that 	$\big|\kappa \ln(|\kappa|)\big|
		>
		L C_{\mathcal{G}}$  and~$|\kappa|\in(0,1)$ holds.
\end{proposition}

\begin{proof}
		As shown in~Proposition~\ref{Corollary1},   the boundary conditions~\eqref{HLyapunov} require the bound ${
		\mub \leq \bar{\mu}\coloneqq
		-\nicefrac{2}{L} \ln\big( {|\kappa|} \big). 
	}$	
	Hence, the weights read as~$w^{\pm}(x;\bar{\mu})
	=
	\kappa^{\frac{2x}{L}}
	\in
	\big[ \kappa^2,1 \big] 
	$ 
	and the Lipschitz continuity of the source term implies
	$$
	\left\lVert
	\mathcal{G}\big(\Riemann(t,x);x\big) 
	\right\rVert_{W(x)}
	\leq
	C_{\mathcal{G}}
	\big\lVert
	\Riemann(t,x) 
	\big\rVert_{2}
	\leq
	\frac{ C_{\mathcal{G}} }{|\kappa|} 
	\big\lVert
	\Riemann(t,x) 
	\big\rVert_{W(x)}.
	$$
	Then, the claim follows from the assumption~$\big|\kappa \ln(|\kappa|)\big|
		>
		L C_{\mathcal{G}}$ and the estimate
		\begin{align*}
			\Riemann(t,x)^\T
			\Mtilde\big(\Riemann(t,x);x,\bar{\mu}\big)
			\Riemann(t,x)
			&=
			\mub\,
			\big\lVert \Riemann(t,x) \big\rVert_{W(x)} 
			+2
			\left\langle
			\Riemann(t,x), 
			\mathcal{G}\big(\Riemann(t,x);x\big)
			\right\rangle_{W(x)} \\
			&\geq
			\left[\bar{\mu}- \frac{ 2C_{\mathcal{G}} }{|\kappa|}  \right]
			\big\lVert
			\Riemann(t,x) 
			\big\rVert_{W(x)}^2
			>0
			\ \ \text{for}\ \
			\Riemann(t,x)\neq 0. 
	\end{align*}	
\end{proof}

\noindent
Finally, we~remark that the regularity assumption~\eqref{ToStrongRegularity}, which has been used to deduce the previous results, can be stated in terms of~$L^2$-solutions for general linear balance laws~\cite[Sec.~2.1.3]{O1} and for  semilinear systems with locally Lipschitz continuous source term~\cite[Th.~1]{Prieur2018}. The following computational framework is based on these $L^2$-solutions, i.e.~$
\u 
\in 
C^0\left(
[0,T);
L^2\big( (0,L);\mathbb{R}^{2n} \big)
\right)
$, 
which allow to consider discontinuities and initial data that may vary widely from steady states.

\section{Computational framework}\label{ComputationalFramework}
Since analytical results are \emph{in general} not available for   semilinear systems, we introduce a computational framework that  is based on a central, weighted, essentially non-oscillatory (CWENO) reconstruction.  The aim is to find a control law such that there exists a parameter~$\mub\geq 0$  that satisfies the conditions~\eqref{A2},~\eqref{A1tilde} and~\eqref{A1}, respectively. 


\subsection{High-order discretization inside the spatial domain}\label{SecCWENO}
A desirable numerical scheme should approximate at high-order not only the semilinear system~\eqref{IBVP1} with space-depending source term and the boundary conditions~\eqref{IBVP2}, but also the conditions~~\eqref{A2},~\eqref{A1tilde} and~\eqref{A1}. 
To this end, we use a finite-volume based CWENO reconstruction.
The spatial domain~$[0,L]$ is divided into $N$~cells~$\mathbb{C}_j\coloneqq\big[x_{j-\nicefrac{1}{2}}, x_{j+\nicefrac{1}{2}} \big]$ for ${j=1,\ldots,N}$ by a space discretization~${\Delta x > 0}$ satisfying ${\Delta x N = L}$. The cell centers are~${x_j \coloneqq (j-\nicefrac{1}{2}) \Delta x}$ and the  cell edges are~${x_{j+{\nicefrac{1}{2}}} \coloneqq j \Delta x}$. 
The evolution of cell averages
$$
\Rbar_j(t)\coloneqq 
\frac{1}{\Delta x}
\int\limits_{x_{j-{\nicefrac{1}{2}}}}^{x_{j+{\nicefrac{1}{2}}}} 
\Riemann(t,x)
\d x
$$
for a general balance law
${\Riemann_t + \tilde{f}(\Riemann)_x=-\tilde{\mathcal{G}}(\Riemann;x)}$
is described by the ordinary differential equation
\begin{equation*}
	\frac{\textup{d}}{\textup{d} t}\Rbar_j(t)=
	- \frac{1}{\Delta x} \bigg[ \tilde{f}\Big( \Riemann(t,x_{j+\nicefrac{1}{2}}) \Big) - \tilde{f}\Big( \Riemann(t,x_{j-\nicefrac{1}{2}})\Big) \bigg]
	-
	\frac{1}{\Delta x}
	\int\limits_{x_{j-{\nicefrac{1}{2}}}}^{x_{j+{\nicefrac{1}{2}}}} 
	\tilde{\mathcal{G}}\Big( \Riemann(t,x);x \Big)
	\d x .
\end{equation*}

\noindent
Here, the linear PDE~\eqref{IBVP1} is written in conservative form by defining
$$
\tilde{f}\Big( \Riemann(t,x) \Big)
\coloneqq
\Lambda(x)  \Riemann(t,x) 
\quad
\text{and}
\quad
\tilde{\mathcal{G}}\Big( \Riemann(t,x);x \Big)
\coloneqq
\mathcal{G}\Big( \Riemann(t,x);x \Big)
-
\partial_x
\Lambda(x)  \Riemann(t,x).
$$
Furthermore, the  central, weighted, essentially non-oscillatory (CWENO) reconstruction from~\cite{Visconti2018} is applied for the interior cells~$j=2,\ldots,N-1$.  
A third-order reconstruction is of the form
\begin{equation}\label{CWENO}
	\textup{CWENO} 
	\ : \ 
	\big[ \Rbar_{j-1},\Rbar_{j},\Rbar_{j+1} \big]
	\ \mapsto \
	\poly_j(x),
\end{equation}
where $\poly_j(x)$ denotes a  reconstruction polynomial defined for~$x\in\mathbb{C}_j$. 
The reconstruction for the semi-discretization at the right ($r$) side of a cell interface is denoted by~${
	\caSD^r_{j-\nicefrac{1}{2}}(t)
	\coloneqq
	\poly_j(x_{j-\nicefrac{1}{2}})
}$, 
at the left~($\ell$) side by~${
	\caSD^\ell_{j+\nicefrac{1}{2}}(t)
	\coloneqq
	\poly_j(x_{j+\nicefrac{1}{2}})
}$ and at the cell center~($c$) by~${
	\caSD^c_{j}(t)
	\coloneqq
	\poly_j(x_{j})
}$. 
The source term is discretized by the Gauss-Lobatto rule with three quadrature nodes.
Then, the resulting semi-discretization of the balance law with the upwind flux~$\NumFlux$ reads as
\begin{align*}
	\frac{\textup{d}}{\textup{d} t}\Rbar_j(t)=
	&- \frac{1}{\Delta x} \bigg[ \NumFlux \Big( \caSD^\ell_{j+\nicefrac{1}{2}}(t),\caSD^r_{j+\nicefrac{1}{2}}(t) \Big) 
	- 
	\NumFlux\Big( 
	\caSD^\ell_{j-\nicefrac{1}{2}}(t),\caSD^r_{j-\nicefrac{1}{2}}(t) 
	\Big)  \Bigg] \\
	&-
	\frac{1}{6}
	\Big[
	\mathcal{G} \big( \caSD^r_{j-\nicefrac{1}{2}}(t) ;x_{j-\nicefrac{1}{2}}\big)
	+ 4
	\mathcal{G} \big( \caSD^c_{j}(t) ;x_j\big)
	+
	\mathcal{G} \big( \caSD^\ell_{j+\nicefrac{1}{2}}(t) ;x_{j+\nicefrac{1}{2}}\big)
	\Big] \\
	&+\  \mathcal{O}\big(\Delta x^3\big).
\end{align*}
It is approximated in time with a strong stability preserving (SSP) Runge-Kutta
method with three stages~\cite{Jiang1996}.

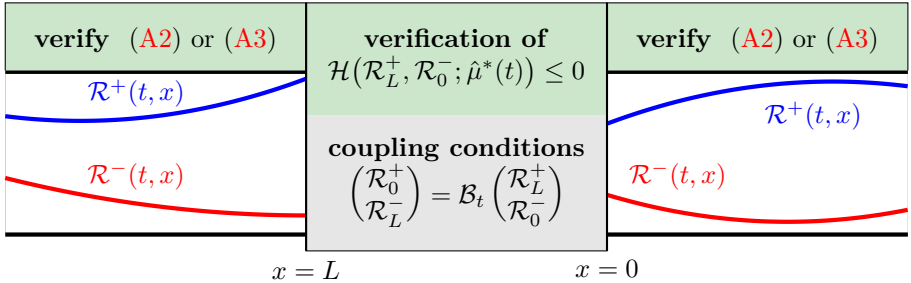
\begin{figure}[H]
	\begin{center}

\tikzstyle{block} = [thick,draw, fill=blue!20, rectangle, 
minimum height=3em, minimum width=3em]
\tikzstyle{Kreis} = [thick,draw, fill=blue!20, circle, 
minimum height=3em, minimum width=3em]
\tikzstyle{sum} = [draw, fill=darkgreen!20, rectangle, 
minimum height=2em, minimum width=3em]
\tikzstyle{input} = [coordinate]
\tikzstyle{output} = [coordinate]
\tikzstyle{pinstyle} = [ pin edge={to-,thick,black}]
\tikzstyle{vecArrow} = [ultra thick,double distance=60pt]

\hspace*{-1.4cm}
\begin{tikzpicture}[scale=1]

	\draw [vecArrow] (0,0) -- (4,0);
	\draw [vecArrow] (8,0) -- (12,0);
	
	\begin{scope}[xshift=-0.4cm,yshift=-1.7cm]
		
		\begin{axis}[scale=0.7,
			domain=-0.5:0.5,
			xticklabels={,,},
			yticklabels={,,},
			samples=100,
			ymin=-2,ymax=2,
			hide x axis,
			hide y axis,    
			axis line style={draw opacity=0}],
			
			\addplot [color=blue, ultra thick] {      x^2 + 0.5*x + 0.2}; 
			\addplot [color=red, ultra thick] {      0.5*x^2 + -0.5*x + -1}; 
		\end{axis}
		
	\end{scope}
	
	\begin{scope}[xshift=7.6cm,yshift=-1.7cm]
		
		\begin{axis}[scale=0.7,
			domain=-0.5:0.5,
			xticklabels={,,},
			yticklabels={,,},
			samples=100,
			ymin=-2,ymax=2,
			hide x axis,
			hide y axis,    
			axis line style={draw opacity=0}],
			
			\addplot [color=blue, ultra thick] { -x^2 + 0.5*x + 0.6}; 
			\addplot [color=red, ultra thick] { 1*x^2 - 0.2*x - 1.2}; 
		\end{axis}
		
	\end{scope}

	\filldraw[fill=gruen!20]  (4,0.5) -- (4,2) -- (8,2) -- (8,0.5);
	\filldraw[fill=black!10]  (4,0.5) -- (4,-1.3) -- (8,-1.3) -- (8,0.5);
	\filldraw[fill=gruen!20]  (0,+1.1) -- (0,+2) -- (4,+2) -- (4,+1.1);
	\filldraw[fill=gruen!20]  (8,+1.1) -- (8,+2) -- (12,+2) -- (12,+1.1);
	
	\node [align=center] at (6,-0.4) {\textbf{coupling conditions}\\
		$
		\begin{pmatrix} \Riemann^+_0 \\ \Riemann^-_L \end{pmatrix} 
		=
		\mathcal{B}_t
		\begin{pmatrix} \Riemann^+_L \\ \Riemann^-_0 \end{pmatrix} 
		$};
	
	\node [align=center] at (6,1.26) {\textbf{verification of}\\
		$
		\H\big(\Riemann^+_L,\Riemann^-_0;\mut(t)\big) \leq 0 
		$};
	
	\node [align=center] at (2,1.5) {\textbf{verify }
		\eqref{A1tilde}~or~\eqref{A1}};
	
	\node [align=center] at (1.8,0.8) {\color{blue} $\Riemann^+(t,x)$ };
	\node [align=center] at (1.8,-0.3) {\color{red} $\Riemann^-(t,x)$ };

	\node [align=center] at (10.8,0.5) {\color{blue} $\Riemann^+(t,x)$ };
	\node [align=center] at (9,-0.3) {\color{red} $\Riemann^-(t,x)$ };
	
	\node [align=center] at (10,1.5) {\textbf{verify }
		\eqref{A1tilde}~or~\eqref{A1}};

	\draw[thick] (4,2) -- (4,-1.3)		node[below] {$x=L$};
	\draw[thick] (8,2) -- (8,-1.3)		node[below] {$x=0$};

\end{tikzpicture}

		\caption{Continuous setting for the computation of stabilizing boundary controls~$\mathcal{B}_t$.}
		\label{FigContinuous}
	\end{center}
\end{figure}

\subsection{High-order discretization at boundaries}
Figure~\ref{FigContinuous} illustrates the continuous setting to obtain stabilizing boundary controls.  
As  mentioned in~Remark~\ref{Remark1}, 
the parameter~$\mub\geq 0$ should be chosen as \emph{small} as possible, since a higher value restricts the choice of boundary conditions. To ensure a  given decay rate~$\mu>0$ we define the possibly time-depending parameter ${
	\mut(t) \in \big\{
	\mut_{\M}(t),\mut_{\Rayleigh}(t)
	\big\}
}$ 
by
\begin{equation}\label{Defmut}
	\begin{aligned}
		\mut_{\M}(t)
		&\coloneqq
		\inf
		\Big\{
		\mub>0 \ \Big| \
		\smallestEV\left\{
		\M\big(\Riemann(t,x);x;\mub\big)
		\right\}
		\geq \mu>0 \text{ for all }x\in[0,L]
		\Big\}, \\
		\mut_{\Rayleigh}(t)
		&\coloneqq
		\inf
		\Big\{
		\mub>0 \ \Big| \
		\Rayleigh\big[
		\Riemann,\mub
		\big](t)
		\geq
		\mu> 0
		\Big\}.
	\end{aligned}
\end{equation}


\noindent
Since this parameter enters the weights $W(x;\mut(t))$ of the Lyapunov function, the condition~$\H\big(\Riemann^+_L,\Riemann^-_0;\mut(t)\big)\leq 0$ is time-dependent as well. This may require also time-dependent coupling conditons~$\mathcal{B}_t$.

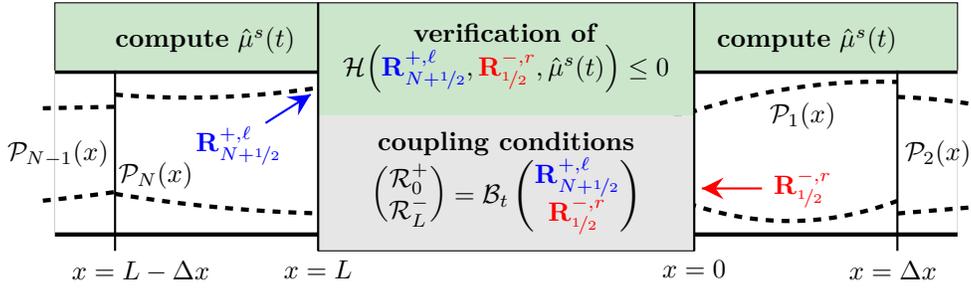
\begin{figure}[H]
	\begin{center}

\tikzstyle{block} = [thick,draw, fill=blue!20, rectangle, 
minimum height=3em, minimum width=3em]
\tikzstyle{Kreis} = [thick,draw, fill=blue!20, circle, 
minimum height=3em, minimum width=3em]
\tikzstyle{sum} = [draw, fill=darkgreen!20, rectangle, 
minimum height=2em, minimum width=3em]
\tikzstyle{input} = [coordinate]
\tikzstyle{output} = [coordinate]
\tikzstyle{pinstyle} = [ pin edge={to-,thick,black}]
\tikzstyle{vecArrow} = [ultra thick,double distance=60pt]

\hspace*{-7mm}
\begin{tikzpicture}[scale=1]

	\draw [vecArrow] (0,0) -- (3.5,0);
	\draw [vecArrow] (8.5,0) -- (12,0);
	
	\begin{scope}[xshift=-0.4cm,yshift=-1.7cm]
		
		\begin{axis}[scale=0.7,
			domain=-0.25:0.45,
			xticklabels={,,},
			yticklabels={,,},
			samples=100,
			ymin=-2,ymax=2,
			xmin=-0.5,xmax=0.5, 
			hide x axis,
			hide y axis,    
			axis line style={draw opacity=0}],
			
			\addplot [color=black, ultra thick, dashed] {      1*x^2 + 0.1*x + 0.45}; 
			\addplot [color=black, ultra thick,dashed] {      0.6*x^2 + -0.5*x + -1}; 
		\end{axis}
		
	\end{scope}
	
	\begin{scope}[xshift=-0.4cm,yshift=-1.7cm]
		
		\begin{axis}[scale=0.7,
			domain=-0.45:-0.25,
			xticklabels={,,},
			yticklabels={,,},
			samples=100,
			ymin=-2,ymax=2,
			xmin=-0.5,xmax=0.5, 
			hide x axis,
			hide y axis,    
			axis line style={draw opacity=0}],
			
			\addplot [color=black, ultra thick, dashed] {       + 0.1*x + 0.35}; 
			\addplot [color=black, ultra thick,dashed] {      0.6*x^2 + 1*x + -0.6}; 
		\end{axis}
		
	\end{scope}
	
	\begin{scope}[xshift=7.6cm,yshift=-1.7cm]
		
		\begin{axis}[scale=0.7,
			domain=-0.45:0.25,
			xticklabels={,,},
			yticklabels={,,},
			samples=100,
			ymin=-2,ymax=2,
			xmin=-0.5,xmax=0.5,
			hide x axis,
			hide y axis,    
			axis line style={draw opacity=0}],
			
			\addplot [color=black, ultra thick,dashed] { -1.5*x^2 + 0.6*x + 0.6}; 
			\addplot [color=black, ultra thick,dashed] { 3*(x+0.049)^2 - 0*x - 1.2}; 
		\end{axis}
		
	\end{scope}

	\begin{scope}[xshift=7.6cm,yshift=-1.7cm]
		
		\begin{axis}[scale=0.7,
			domain=0.25:0.45,
			xticklabels={,,},
			yticklabels={,,},
			samples=100,
			ymin=-2,ymax=2,
			xmin=-0.5,xmax=0.5,
			hide x axis,
			hide y axis,    
			axis line style={draw opacity=0}],
			
			\addplot [color=black, ultra thick,dashed] { -1.5*x^2 + 0.6*x + 0.4}; 
			\addplot [color=black, ultra thick,dashed] {  0.4*x - 1.2}; 
		\end{axis}
		
	\end{scope}

	c
	\filldraw[fill=gruen!20]  (3.5,0.5) -- (3.5,2) -- (8.5,2) -- (8.5,0.5);
	\filldraw[fill=black!10]  (3.5,0.5) -- (3.5,-1.3) -- (8.5,-1.3) -- (8.5,0.5);
	\filldraw[fill=gruen!20]  (0,+1.1) -- (0,+2) -- (3.5,+2) -- (3.5,+1.1);
	\filldraw[fill=gruen!20]  (8.5,+1.1) -- (8.5,+2) -- (12,+2) -- (12,+1.1);
	
	\node [align=center] at (6,-0.4) {\textbf{coupling conditions}\\
		$
		\begin{pmatrix} \Riemann^+_0 \\ \Riemann^-_L \end{pmatrix} 
		=
		\mathcal{B}_t
		\begin{pmatrix}{ \color{blue} \caSD^{+,\ell}_{N+\nicefrac{1}{2}}} \\ {\color{red} \caSD^{-,r}_{\nicefrac{1}{2}} } \end{pmatrix} 
		$};
	
	\node [align=center] at (6,1.26) {\textbf{verification of}\\
		$
		\H\Big({\color{blue} \caSD^{+,\ell}_{N+\nicefrac{1}{2}}},{\color{red} \caSD^{-,r}_{\nicefrac{1}{2}} }, \mud(t) \Big)\leq 0
		$};
	
	\node [align=center] at (2,1.5) {\textbf{compute}
		$\mud(t)$};

	\node [align=center] at (2.5,0.1) {\color{blue} $\caSD^{+,\ell}_{N+\nicefrac{1}{2}}$ };

	
	\node [align=center] at (11.8,0) {\color{black} $\poly_{2}(x)$ };
	\node [align=center] at (0.1,0) {\color{black} $\poly_{N-1}(x)$ };

	\node [align=center] at (1.4,-0.3) {\color{black} $\poly_N(x)$ };
	\node [align=center] at (10,0.5) {\color{black} $\poly_1(x)$ };
	
	\node [align=center] at (10,-0.47) {\color{red} $\caSD^{-,r}_{\nicefrac{1}{2}}$ };
	
	
	\draw[thick,red,decoration={markings,mark=at position 1 with	{\arrow[scale=2,>=stealth]{>}}},postaction={decorate}]  (9.4,-0.47) -- (8.6,-0.47);

	\draw[thick,blue,decoration={markings,mark=at position 1 with	{\arrow[scale=2,>=stealth]{>}}},postaction={decorate}]  (2.8,0.4) -- (3.4,0.78);
	

	\node [align=center] at (10,1.5) {\textbf{compute}
		$\mud(t)$};

	\draw[thick] (3.5,2) -- (3.5,-1.3)		node[below] {$x=L$};
	\draw[thick] (8.5,2) -- (8.5,-1.3)		node[below] {$x=0$};
	\draw[thick] (0.8,1.1) -- (0.8,-1.3)		node[below] {\ \; \quad $x=L-\Delta x$};
	\draw[thick] (11.2,1.1) -- (11.2,-1.3)		node[below] {$x=\Delta x$  \quad};

\end{tikzpicture}

		\caption{Semi-discrete setting for the computation of stabilizing boundary controls~$\mathcal{B}_t$.}
		\label{FigDiscrete}
	\end{center}
\end{figure}

\noindent
Figure~\ref{FigDiscrete} illustrates the discretized setting. Therein, the reconstruction polynomials~$\poly_j(x)$ are black dotted. Since the CWENO reconstruction~\eqref{CWENO} requires the central stencil~$\mathbb{C}_{j-1},\ldots,\mathbb{C}_{j+1}$ to reconstruct the polynomial~$\poly_j(x)$ it can be only used for interior cells. The cells~$\mathbb{C}_0, \mathbb{C}_N$, which are adjacent to the boundary, need a special treatment. Therein,   reconstructions of the form
\begin{alignat*}{8} 
	&\textup{CWENOb}^{(r)}
	\ : \ 
	\big[ \Rbar_{N-2},\Rbar_{N-1},\Rbar_{N} \big]
	&&\ \mapsto \
	\poly_N(x), \\
	&\textup{CWENOb}^{(\ell)} 
	\ : \ 
	\big[ \Rbar_{1},\Rbar_{2},\Rbar_{3} \big]
	&&\ \mapsto \
	\poly_1(x)
\end{alignat*}
are applied~that have been recently introduced by~\cite[Semplice, Travaglia,  Puppo]{Elena}. 
A crucial property of these  reconstruction are polynomials that  are defined within the whole cell. This  allows to determine the expressions~\eqref{Defmut} at high order, i.e.
\begin{align*}
	\mud_{\M}(t)
	&\coloneqq
	\inf
	\Big\{
	\mub>0 \ \Big| \
	\smallestEV\left\{
	\Mtilde\big(\poly_j(x_k);x_k;\mub\big)
	\right\} \geq \mu \\
	&\qquad\quad
	\text{ for all } \ j=1,\ldots,N, \ 
	\text{ and }
	k\in\{j\pm\nicefrac{1}{2},j \}
	\Big\}, \\
	\mud_{\Rayleigh}(t)
	&\coloneqq
	\inf
	\bigg\{
	\mub>0 \ \bigg| \
	\sum\limits_{j=1}^N
	\RayleighQ\big[
	\poly_j,\mub
	\big](t)
	\geq
	\mu> 0
	\bigg\},
\end{align*}
where~$\RayleighQ\big[
\poly_j,\mub
\big]$ denotes the discretized weighted Rayleigh quotient~\eqref{Rayleigh} that is obtained by the Gauss-Lobatto  rule with the quadrature nodes~$
\poly_j(x_{j\pm \nicefrac{1}{2} })
$ and~$
\poly_j(x_{j })
$.


Furthermore, the boundary values~$\caSD^{+,\ell}_{N+\nicefrac{1}{2}}  $ and~$\caSD^{-,r}_{\nicefrac{1}{2}}  $ are reconstructed at high order. This is crucial to verify the condition~$\H\big( \caSD^{+,\ell}_{N+\nicefrac{1}{2}},\caSD^{-,r}_{\nicefrac{1}{2}};\mud(t) \big)\leq 0$. 
The  inflow reconstructed values are given by the coupling conditions, i.e.~$
\caSD^{+,\ell}_{\nicefrac{1}{2}}
\coloneqq \Riemann_0^+$ and
$ 
\caSD^{-,r}_{N+\nicefrac{1}{2}} 
\coloneqq \Riemann_L^-. 
$ 

This framework guarantees  parameters~$\mud(t)$ that satisfy the properties~\eqref{A1tilde},~\eqref{A1} and specify  the condition~\eqref{A2}, which ensures dissipative boundary controls. Those can be specified problem and time-dependent, which ensures stabilizing boundary conditions. However, we note  that 
	there do not necessarily exist 
boundary controls with the property~\eqref{A2}. 
 Systems that are not stabilizable, see~e.g.~\cite{Bastin2011,Gugat2019,GugatHerty2020}, serve as examples.  
Hence, only sufficient conditions are obtained, but not necessarily a control rule that steers the system exponentially fast to a desired state.

\begin{remark}\label{RemarkII}
	The presented computational framework is consistent with the previously mentioned theoretical state of research apart from the fact that   the choices~\eqref{Defmut} lead to a numerical Lyapunov function whose parameters are time-dependent, in contrast to the analytical Lyapunov function~\eqref{analyticalLF}. Therefore, the numerical Lyapunov function may be~\emph{not strictly decreasing}, but still leads to an exponential decay of perturbations. Analytical results are recovered by setting~$\mud(t)=\mub$~for all~$t\in\mathbb{R}^+_0$. 
	
	Furthermore, the parameter~$\mud_{\Rayleigh}(t)$, which is based on the weighted Rayleigh quotient, leads to a weaker estimate of the parameter~$\mub$ that enters the Lyapunov function, i.e.~a smaller parameter~$\mub$ is obtained. Hence, the disadvantage of this choice is a control rule that may not be sufficient to establish exponential decay. On the other hand, perturbations from steady states are still damped and  the computational cost is reduced significantly. Namely, the choice~$\mud_{\M}(t)$ requires to solve an optimization problem for each point in space, whereas the choice~$\mud_{\Rayleigh}(t)$ requires the solution of only~\emph{one} optimization problem.
\end{remark}

\section{Numerical results}\label{SecNumericalResults}

The  computational framework is applicable for general systems of semilinear hyperbolic balance laws and will be illustrated by means of the Kac-Goldstein equations for chemotaxis with boundary conditions~\eqref{BCKAC}. Since these systems are in general not stabilizable~\cite{Bastin2011,Gugat2019,GugatHerty2020}, we consider the benchmark problems derived in Section~\ref{SectionConditions}.
More precisely, we consider the Kac-Goldstein equations with the following source terms:
\begin{alignat*}{8}
&\textup{Section~\ref{SecTurningI}:}\quad
&&
\mathcal{G} \big( \Riemann(t,x);x \big) 
=
\frac{1}{2e} 
\begin{pmatrix}
1 & -1 \\
-1 & 1
\end{pmatrix}
\Riemann(t,x) \\
&\textup{Section~\ref{SecTurningII}:}\quad
&&
\mathcal{G} \big( \Riemann(t,x);x \big) 
=
\frac{1}{2e} 
\cos\Big(
\Riemann^+(t,x)^2
+
\Riemann^-(t,x)^2
\Big)
\begin{pmatrix}
1 & -1 \\
-1 & 1
\end{pmatrix}
\Riemann(t,x)\\
&\textup{Section~\ref{SecTurningIII}:}\quad
&&
\mathcal{G} \big( \Riemann(t,x);x \big) 
=
\frac{1}{2e} 
\begin{pmatrix}
\Riemann^-(t,x) & \Riemann^+(t,x) \\
\Riemann^+(t,x) & \Riemann^-(t,x)
\end{pmatrix}
\begin{pmatrix}
1 & -1 \\
-1 & 1
\end{pmatrix}
\Riemann(t,x) \\
\end{alignat*}

\noindent
The  nonlinear source term  requires the computational approach of  Section~\ref{ComputationalFramework}. It is a relevant practical example~to model the alignment  of animals
and cells~\cite[Sec.~2]{Lutscher}. Although a small value~$|\kappa|\ll 1$ may damp perturbations faster, a large value is typically of interest. Namely,  larger values can lead to a smaller net flux at the boundary, which in turn results in a more economical control. 
The discretizations
$\Delta x=2^{-5}$ and $\textup{CFL}=0.45$ for~$L=1$ are used in all simulations.  Initial conditions read as~$\Riemann^+(0,x)=\sin(\pi x)$ and~$\Riemann^-(0,x)=\cos(\pi x)$. 
Furthermore, parameters for the CWENO reconstructions are chosen as $Z$-nonlinear weights~as described in~\cite{Elena}.

\subsection{Linear source term}\label{SecTurningI}
The linear source term is a special case of Proposition~\ref{Corollary1} for~${ \theta(x)=\nicefrac{1}{e}}$. 
 The largest parameter for the boundary conditions that is guaranteed in Proposition~\ref{Corollary1} is~$\kappa=\nicefrac{1}{e}$, which is obtained for~$
\mub=
-\nicefrac{2}{L} \ln\big( {|\kappa|} \big)=2
$.

Figure~\ref{Corollary_1_Solution} shows the deviation~$\rho(t,x)-\bar{\rho}$ of the density for~$t\in[0,10]$. If the time-dependent parameter~$\mud_{{\M}(t)}$ is used~(left panels), the obtained boundary control is able to steer the system towards the equilibrium for both decay rates~$\mu=0.1$ and~$\mu=1$. In contrast, the parameter~$\mud_{\mathcal{Q}}(t)$, which is based on the weighted Rayleigh quotient only stabilizes the system provided that the larger desired decay rate~$\mu=1$ is used. 
Figure~\ref{Corollary_1_Lyapunov}  shows the corresponding Lyapunov functions that yield upper bounds for the $L^2$-norm of deviations. More precisely, the scaled Lyapunov function
\begin{equation}\label{LFtilde}
	\widetilde{\L}(t)
	\coloneqq
	W^{-1}_{\min}(t) \L(t) 
	\geq \big\lVert \Riemann \big\rVert_{L^2}
	\ \text{ for } \
	W_{\min}(t)
	\coloneqq
	\min_{x\in[0,L]} \bigg\{
	\smallestEV\Big\{
	W\big(x,\mud(t)\big)
	\Big\}
	\bigg\}
\end{equation}
is shown as blue line and serves as upper bound for the $L^2$-norm that is obtained by simulations with the parameter~$\mud_{{\M}(t)}$, which is shown in blue. They decay exponentially fast with a rate that is larger than the desired  rate~$\mu=0.1$~(left panel) and $\mu=1$~(right panel). 
The $L^2$-norm for deviations that are based on the~parameter~$\mud_{{\mathcal{Q} }(t)}$ is shown in red. 
Deviations even increase for the smaller desired decay rate and decrease for~$\mu=1$, however, with a lower rate.

\begin{figure}[H]
	\begin{minipage}{0.45\textwidth}
		\begin{center}	
			$\boldsymbol{\mud_{\M}(t)}$\textbf{  and }$\boldsymbol{\mu=0.1}$	
		\end{center}
		\vspace{0mm}
		\scalebox{1}{\includegraphics[width=\linewidth]{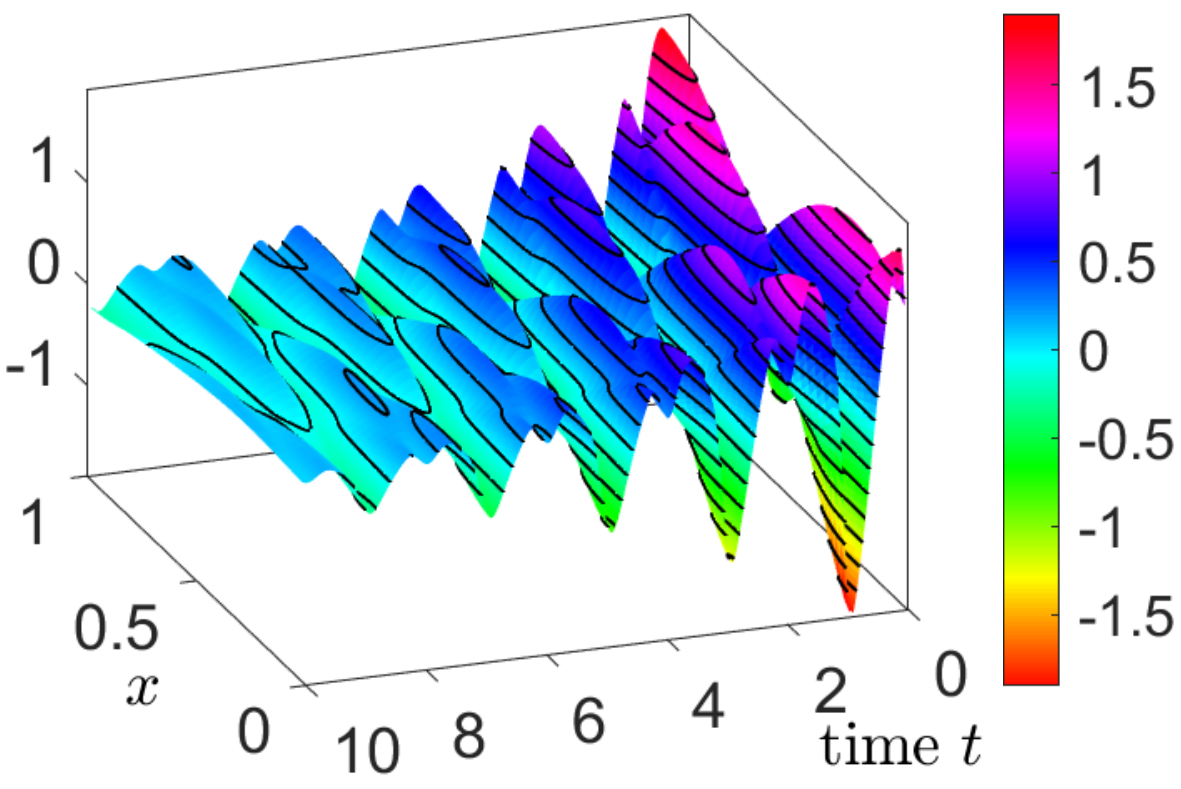}}
	\end{minipage}
	\hfil
	\begin{minipage}{0.45\textwidth}
		\begin{center}	
			$\boldsymbol{\mud_{Q}(t)}$\textbf{  and }$\boldsymbol{\mu=0.1}$	
		\end{center}
		\vspace{0mm}
		\scalebox{1}{\includegraphics[width=\linewidth]{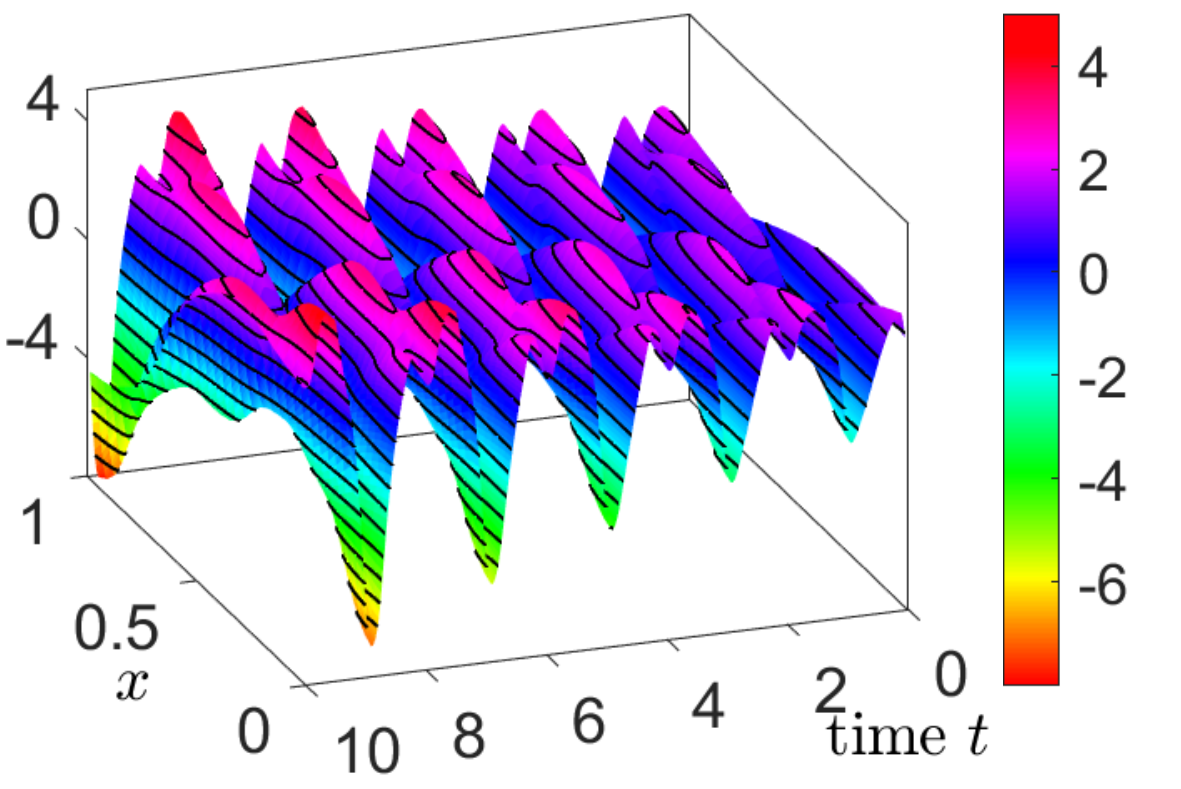}}
	\end{minipage}
	
	
	\bigskip\bigskip
	
	\begin{minipage}{0.45\textwidth}
		\begin{center}	
			$\boldsymbol{\mud_{\M}(t)}$\textbf{  and }$\boldsymbol{\mu=1}$	
		\end{center}
		\vspace{0mm}
		\scalebox{1}{\includegraphics[width=\linewidth]{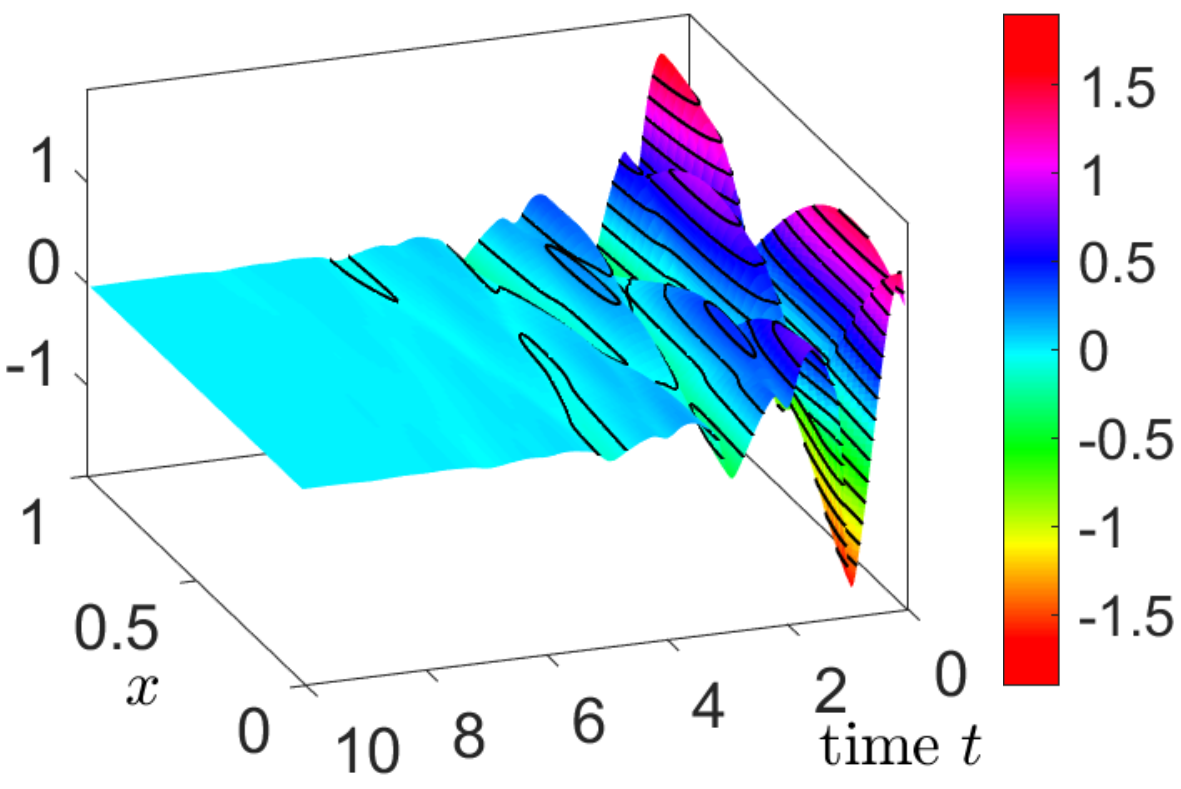}}
	\end{minipage}
	\hfil
	\begin{minipage}{0.45\textwidth}
		\begin{center}	
			$\boldsymbol{\mud_{\mathcal{Q}}(t)}$\textbf{  and }$\boldsymbol{\mu=1}$	
		\end{center}
		\vspace{0mm}	\scalebox{1}{\includegraphics[width=\linewidth]{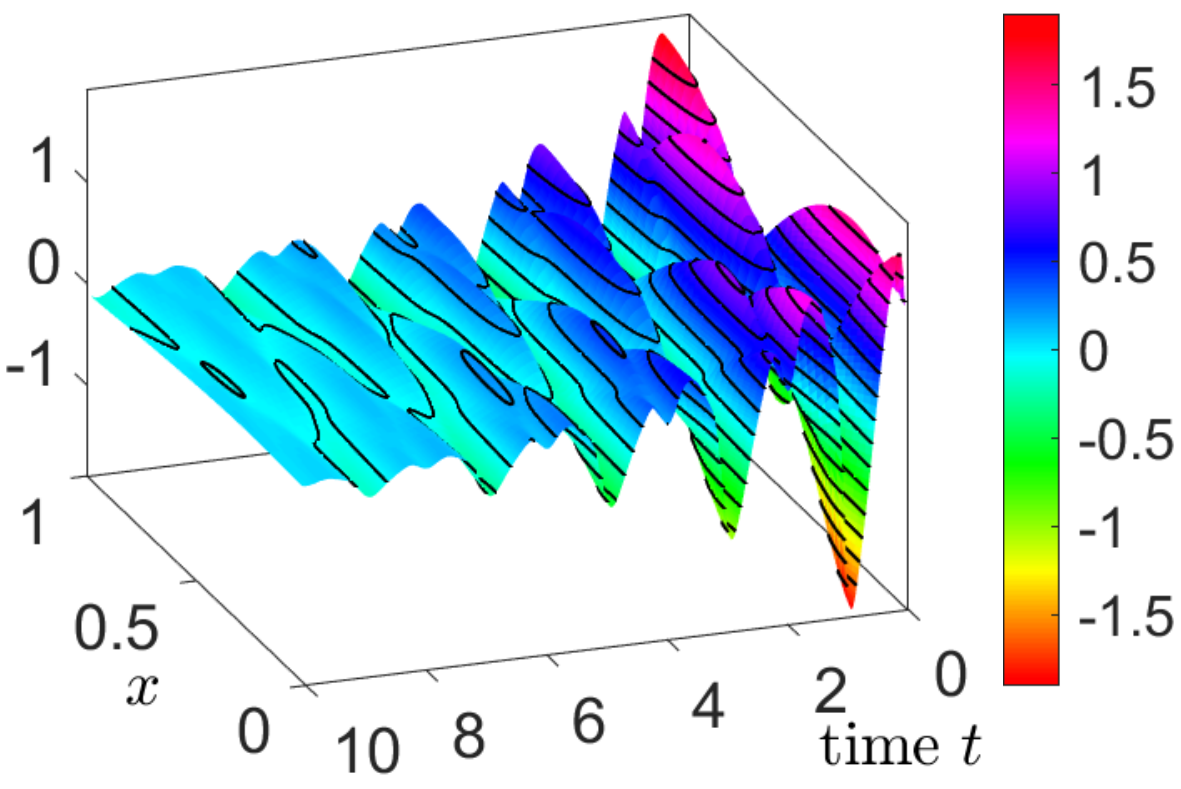}}
	\end{minipage}

	\caption{Deviations of the density to the steady state. The upper panels show a simulation with desired decay rate~$\mu=0.1$ where the parameters~${\mud_{\M}(t)}$~(left) and ${\mud_{\mathcal{Q}}(t)}$ (right) are used. The lower panels show simulations for the decay rate~$\mu=1$, respectively. }
	\label{Corollary_1_Solution}
	
\end{figure}

\begin{figure}[H]
	\begin{minipage}{0.45\textwidth}
		\begin{center}	
			$\boldsymbol{\mu=0.1}$	
		\end{center}
		\scalebox{1}{\includegraphics[width=\linewidth]{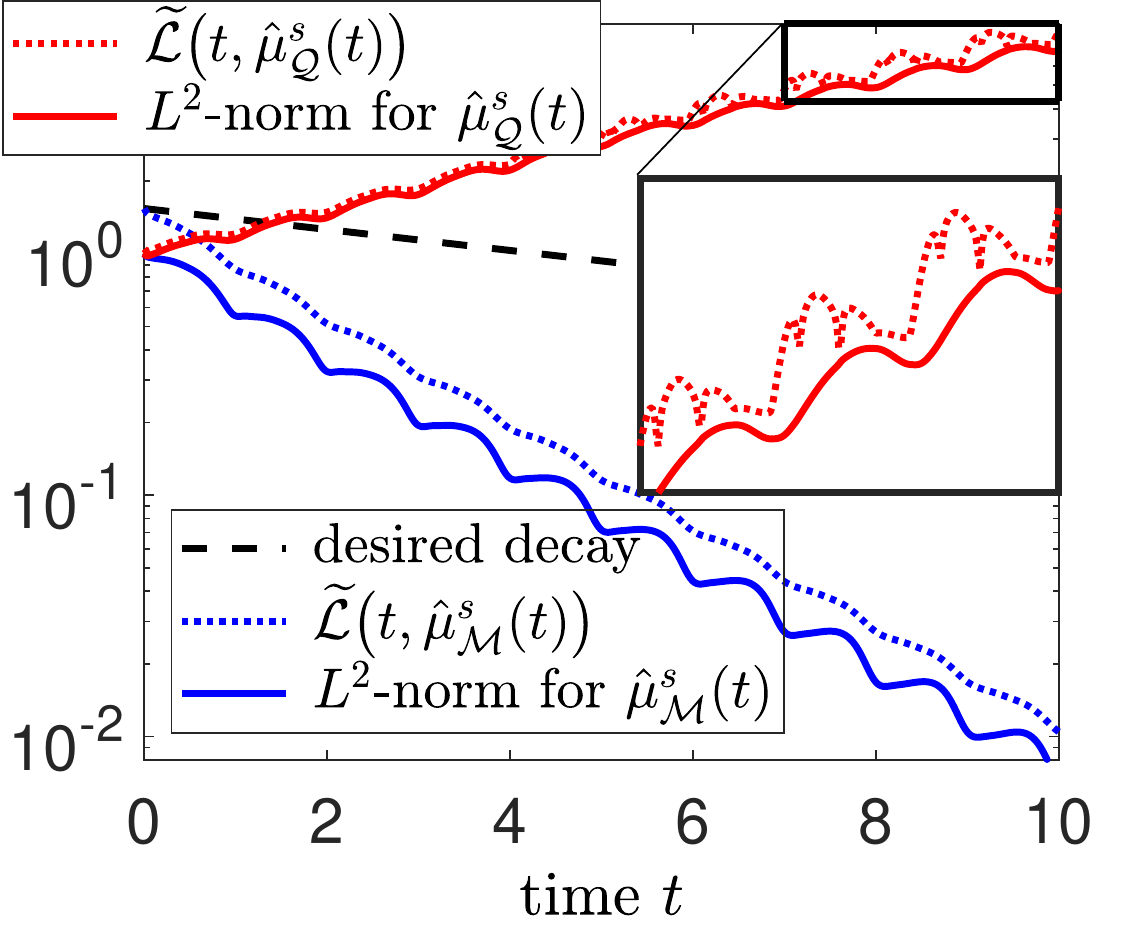}}
	\end{minipage}
	\hfil
	\begin{minipage}{0.45\textwidth}
		\begin{center}	
			$\boldsymbol{\mu=1}$	
		\end{center}
		\scalebox{1}{\includegraphics[width=\linewidth]{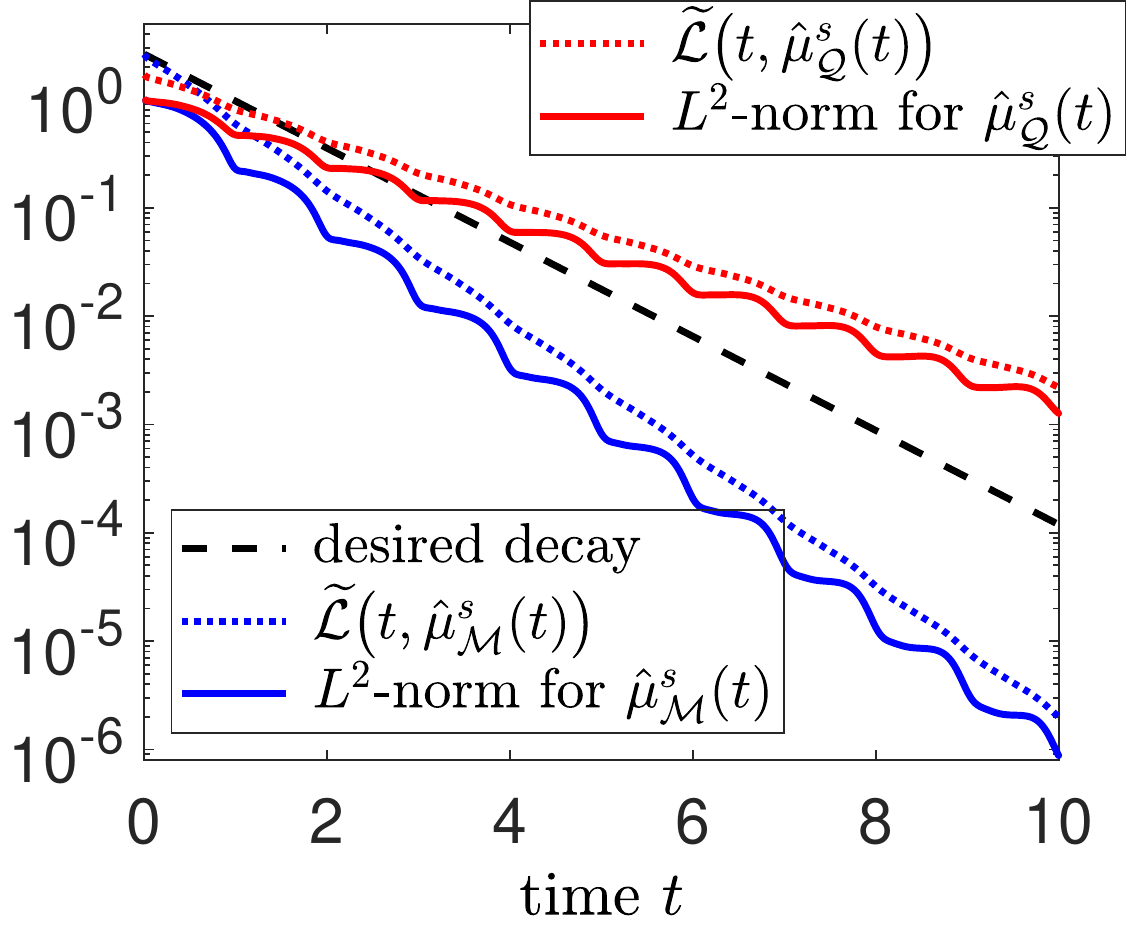}}
	\end{minipage}

	\caption{The~$L^2$-norm~$\big\lVert \Riemann(t,\cdot) \big\rVert^2_{L^2}$ 
		and the scaled Lyapunov function~\eqref{LFtilde}	
		obtained by simulations with the parameter~$\mud_{{\M }(t)}$ are plotted in blue. Simulations that   are based on the parameter~$\mud_{{\mathcal{Q} }(t)}$ are shown in red, respectively. The desired exponential decay is black dashed. }
	\label{Corollary_1_Lyapunov}
\end{figure}

\begin{figure}[H]
	\begin{minipage}{0.45\textwidth}
		\begin{center}	
			$\boldsymbol{\mu=0.1}$	
		\end{center}	
		\scalebox{1}{\includegraphics[width=\linewidth]{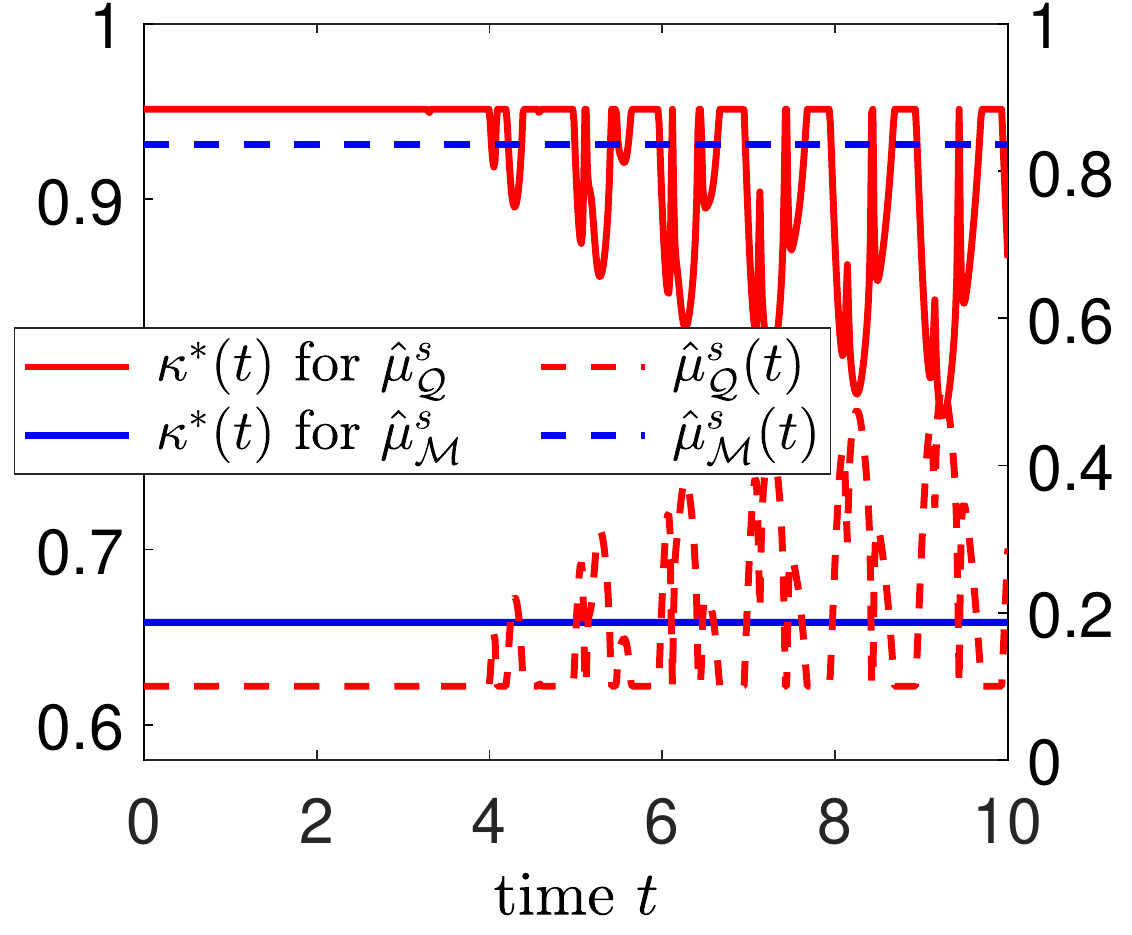}}
	\end{minipage}
	\hfil
	\begin{minipage}{0.45\textwidth}
		\begin{center}	
			$\boldsymbol{\mu=1}$	
		\end{center}
		\scalebox{1}{\includegraphics[width=\linewidth]{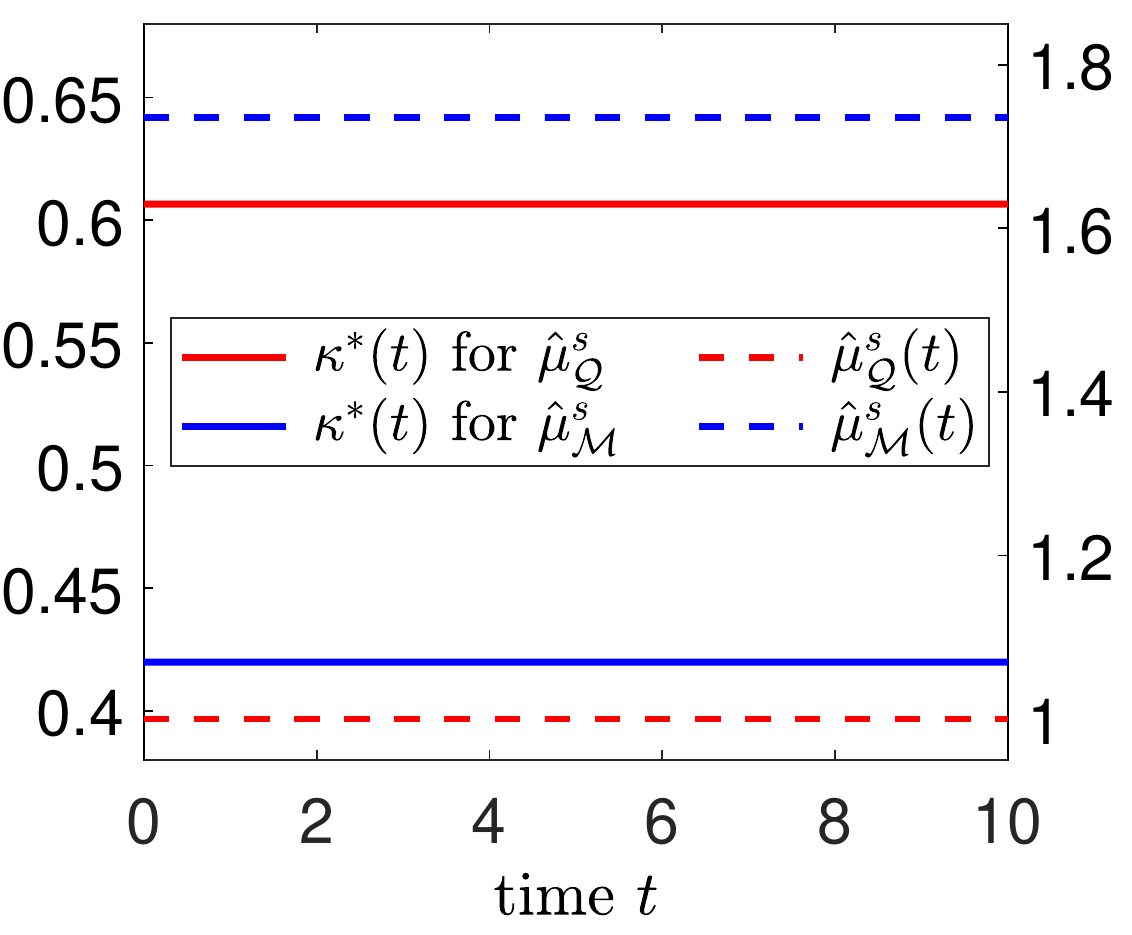}}
	\end{minipage}
	
	\caption{Left~$y$-axis shows the obtained boundary control~$\kappa^*(t)$ for the parameters~$\mud_{{\mathcal{Q} }(t)}$ and $\mud_{{\mathcal{M} }(t)}$ which are plotted against the right axis as dashed line.}
	\label{Corollary_1_BC}	
	
\end{figure}

Figure~\ref{Corollary_1_BC} shows the obtained boundary control~$\kappa^*(t)$ and the parameter~$\mud(t)$, which enters the Lyapunov function.  
	More precisely, 
	the parameters $\mud_{{\M }(t)}$ and $\mud_{{\mathcal{Q} }(t)}$ are dashed plotted against the right~$y$-axis. 
	The resulting control~$\kappa^*(t)$ is shown with respect to the left~$y$-axis. Namely, the control corresponding to the parameter~$\mud_{{\M }(t)}$ 
	is plotted as blue line and  that one corresponding to~$\mud_{{\mathcal{Q} }(t)}$ as red line, respectively. 
We observe that simulations that are based on the weighted Rayleigh quotient yield smaller values~$
\mud_{{\mathcal{Q} }(t)}
\leq
\mud_{{\M }(t)}
$ 
which in turn result in a larger control parameter. Likewise, a smaller desired decay rate leads to larger control parameters. Furthermore, we note that the control based on the Rayleigh quotient, which does not steer the system to the steady state, still counteracts fluctuations. In particular, it actively reduces the control parameters when instabilities increase. 
	A zoom in the left panel of Figure~\ref{Corollary_1_Lyapunov} highlights the non-monotonic behaviour of the scaled Lyapunov function and the $L^2$-norm of deviations when the control parameter~$\mud_{\mathcal{Q} }(t)$ is used. These oscillations arise from fluctuations of the control, which are illustrated in  the left panel of Figure~\ref{Corollary_1_Lyapunov}. In contrast, the~control parameter~$\mud_{\mathcal{M} }(t)$ is almost constant and the resulting Lyapunov function decays monotonically.

\subsection{Lipschitz continuous source term}\label{SecTurningII}

The distance to the equilibrium state is shown in Figure~\ref{Corollary_2_Solution} for the boundary control that is obtained by the parameters~$\mud_{\M}(t)$~(left panel) and~$\mud_{\Q}(t)$~(right panel). 
As expected, the parameter~$\mud_{\M}(t)$ leads to a boundary control that steers the system faster to the steady state.

\begin{figure}[H]
	\begin{minipage}{0.45\textwidth}
		\begin{center}	
			\textbf{perturbations for }$\boldsymbol{\mud_{\M}(t)}$	
		\end{center}
		\vspace{0mm}
		\scalebox{1}{\includegraphics[width=\linewidth]{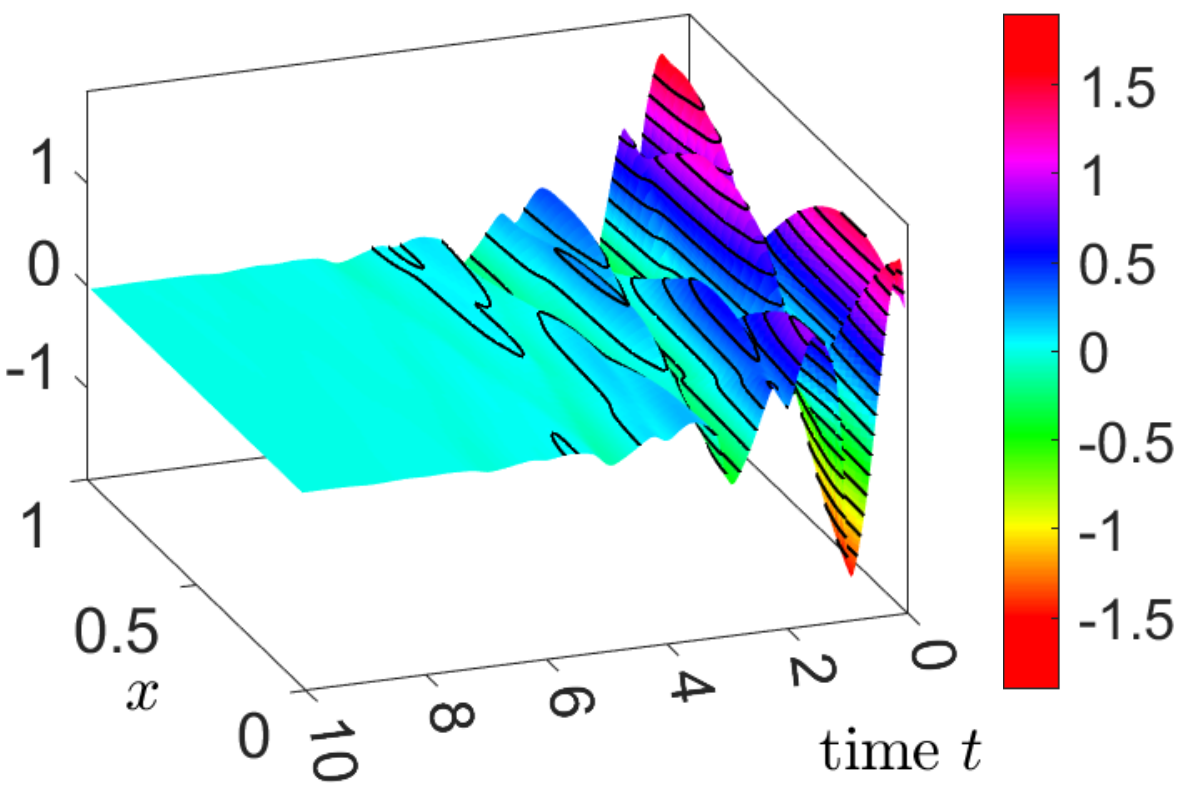}}
	\end{minipage}
	\hfil
	\begin{minipage}{0.45\textwidth}
		\begin{center}	
			\textbf{perturbations for }$\boldsymbol{\mud_{\mathcal{Q}}(t)}$	
		\end{center}
		\vspace{0mm}
		\scalebox{1}{\includegraphics[width=\linewidth]{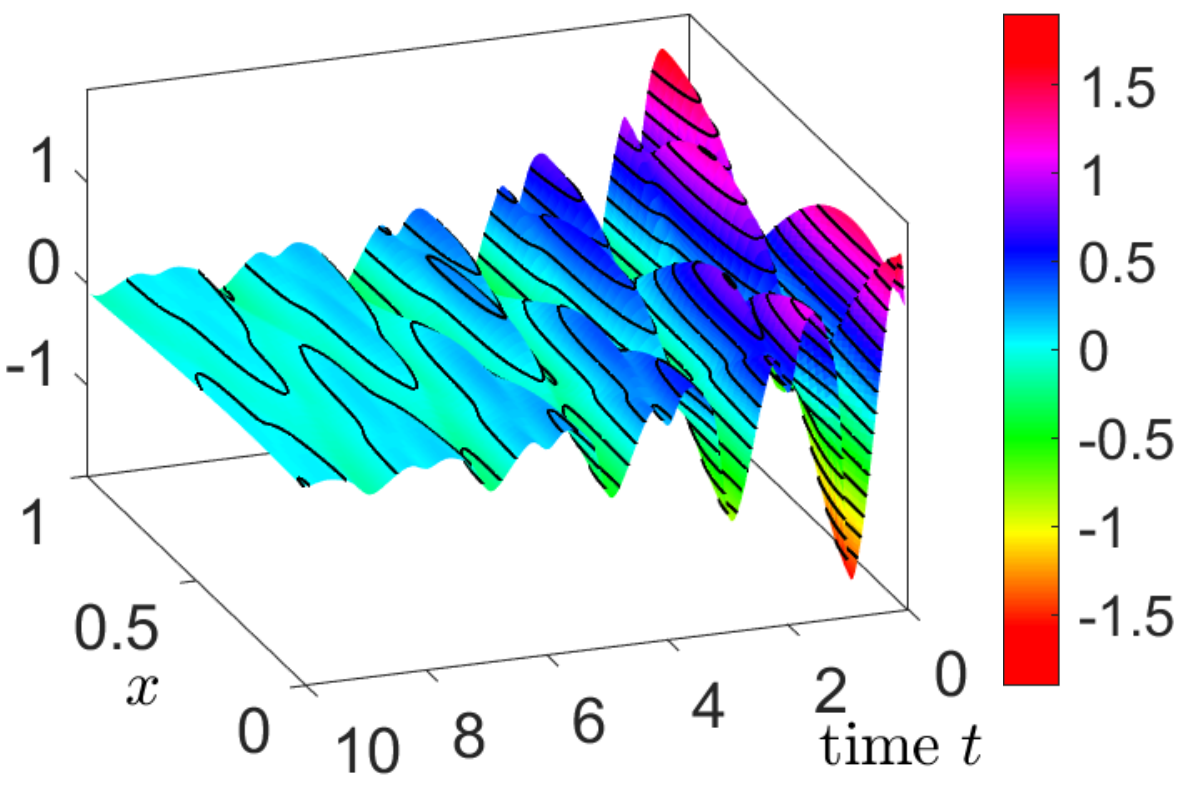}}
	\end{minipage}
	
	\caption{Deviations of the density to the steady state  with desired decay rate~$\mu=1$, where the parameter~${\mud_{\M}(t)}$~(left) and ${\mud_{\mathcal{Q}}(t)}$ (right) are used for the boundary control. }
	\label{Corollary_2_Solution}
	
\end{figure}

\begin{figure}[H]
	\begin{minipage}{0.45\textwidth}
		\scalebox{1}{\includegraphics[width=\linewidth]{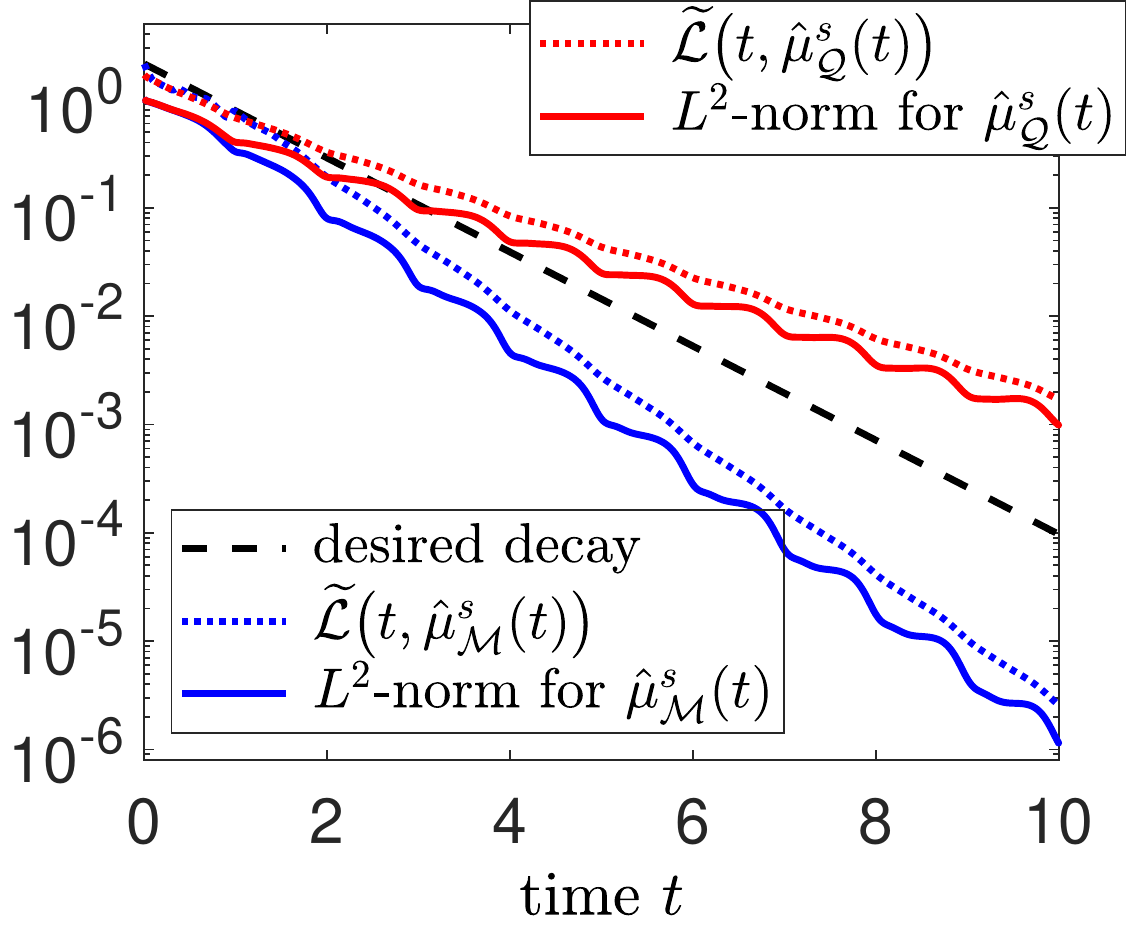}}
	\end{minipage}
	\hfil
	\begin{minipage}{0.45\textwidth}
		\scalebox{1}{\includegraphics[width=\linewidth]{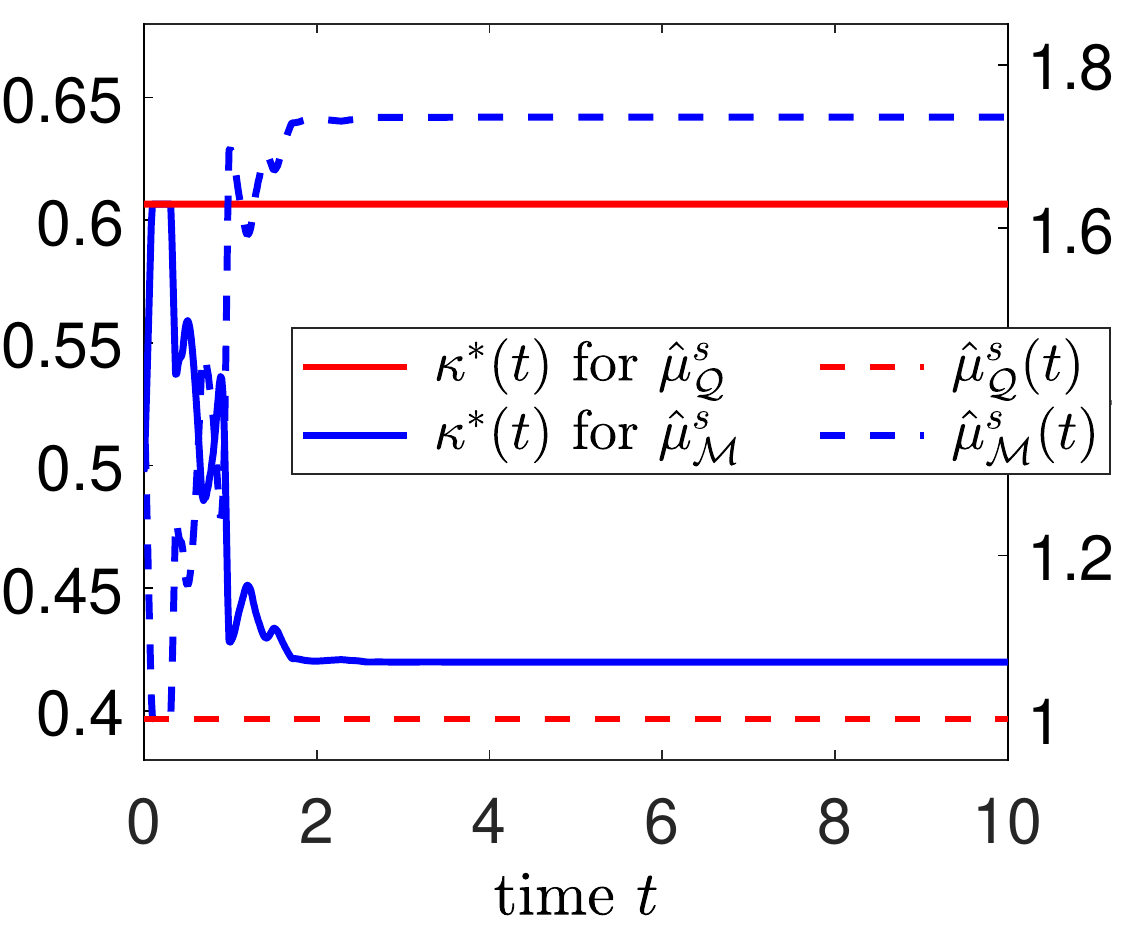}}
	\end{minipage}
	
	\caption{The left panel shows the scaled Lyapunov function~\eqref{LFtilde} and  the~$L^2$-norm  obtained by simulations with the parameter~$\mud_{{\M }(t)}$ in blue, while those corresponding to the  parameter~$\mud_{{\mathcal{Q} }(t)}$ are shown in red.  The desired decay is black dashed. The right panel states the control~$\kappa^*(t)$ at the left~$y$-axis and the corresponding parameters~$\mud_{{\Q }(t)}$ and $\mud_{{\M }(t)}$  at the right axis.}
	\label{Corollary_2_Control}
	
	\color{blue}
	
\end{figure}

The decay is shown in the left panel of Figure~\ref{Corollary_2_Control}. Therein, the $L^2$-norm of the distance to the steady state is shown as blue line for the parameter~$\mud_{\M}(t)$ and in red for the control that is based on the weighted Rayleigh quotient. The scaled~Lyapunov function  yields an upper bound on the $L^2$-norm and decays at least with the desired rate (black, dashed) if the parameter~$\mud_{\M}(t)$ is used. The~$L^2$-norm of deviations that are based on the weighted Rayleigh quotient decays more slowly, but still exponentially fast. 

The right panel of Figure~\ref{Corollary_2_Control} shows the boundary control parameter in the scale of the left~$y$-axis. The weighted Rayleigh quotient leads to a smaller time-dependent value $\mud_{\Q}(t)\leq \mud_{\M}(t)$ and hence to a larger control parameter. Furthermore, the control remains  stable, while large changes in the control occur if the parameter~$\mud_{\M}(t)$ is used.

\subsection{General source term}\label{SecTurningIII}

Figure~\ref{Corollary_3_solution} shows the solution to the nonlinear source term for the decay rate~$\mu=0.1$~(left panel) and~$\mu=1$~(right panel), where the control  based on the parameter~$\mud_{{\M }(t)}$ is used. We observe that the control is able to stabilize the system. More precisely, Figure~\ref{Corollary_3_LF} shows the scaled Lyapunov function~\eqref{LFtilde} and  the $L^2$-norm of deviations. 
	In this particular example, 
	simulations for both $\mud_{{\M }(t)}$ and $\mud_{{\Q }(t)}$ 
	only slightly differ and cannot be distinguished in the plot. 
We observe that the control steers the system to the equilibrium with the desired decay, which is black dashed.  However, the Lyapunov function is~\emph{not strictly decreasing}.

	Figure~\ref{Shock} considers discontinuous initial data~$\Riemann^\pm(0,x)=2\, \textup{sign}(x)$ and a control based on the weighted Rayleigh quotient, i.e.~the 	parameter~$\mud_{{\mathcal{Q} }(t)}$ is used. 	
	For comparison, a simulation without source term is included. The~$L^2$-norm of deviations, described by the resulting conservation law, is shown as blue line. 
	The scaled Lyapunov function~\eqref{LFtilde} yields an upper bound that decays exponentially fast. The observed decay is within the desired rates~$\mu=0.1$ (left panel) and $\mu=1$ (right panel), which are plotted as dashed, black lines. 
	Deviations that result from the semilinear system are shown in red. The scaled Lyapunov function, which serves as upper bound for the $L^2$-norm, decreases over time and makes the $L^2$-norm decay with an asymptotic rate that is similar to the desired decay rate. 
	However, the decay is non-monotone and deviations may be larger than desired.

Hence,  observations from~Figure~\ref{Corollary_3_LF} and Figure~\ref{Shock} 
reflect the fact that the computational framework is beyond the theoretical stabilization concept for~$L^2$-solutions if it is applied to nonlinear source terms.

\begin{figure}[H]
	\begin{minipage}{0.45\textwidth}
		\begin{center}	
			$\boldsymbol{\mu=0.1}$	
		\end{center}	
		\scalebox{1}{\includegraphics[width=\linewidth]{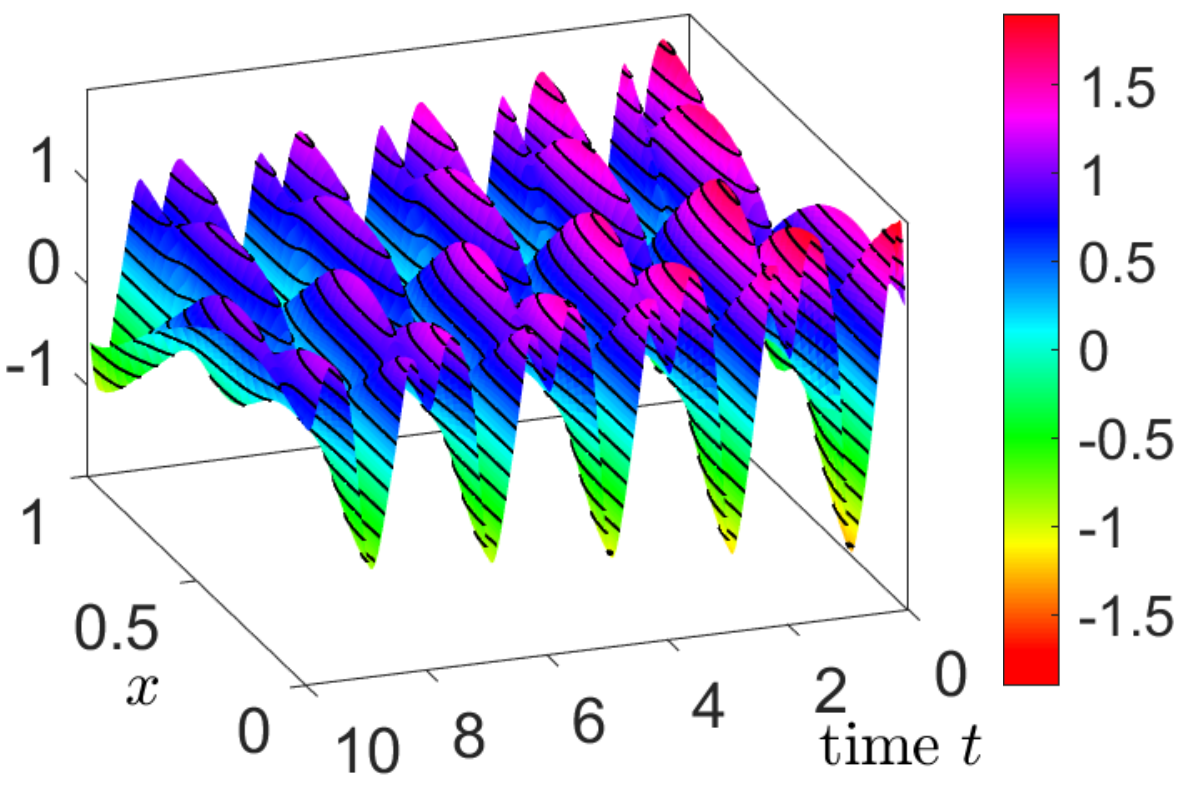}}
	\end{minipage}
	\hfil
	\begin{minipage}{0.45\textwidth}
		\begin{center}	
			$\boldsymbol{\mu=1}$	
		\end{center}
		\scalebox{1}{\includegraphics[width=\linewidth]{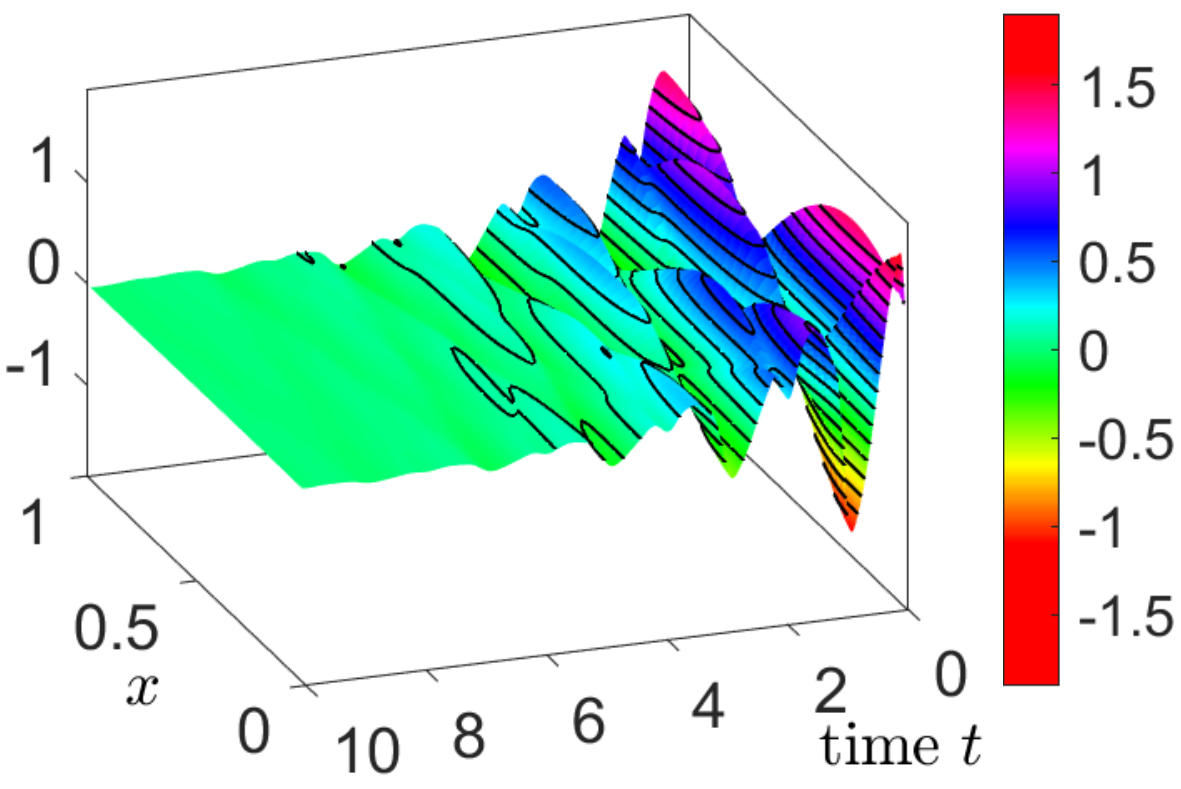}}
	\end{minipage}
	
	\caption{Deviations of the density from steady state for the nonlinear systems with the control based on the parameter~$\mud_{{\M }(t)}$.}
	\label{Corollary_3_solution}	
	
\end{figure}

\begin{figure}[H]
	\begin{minipage}{0.45\textwidth}
		\begin{center}	
			$\boldsymbol{\mu=0.1}$	
		\end{center}	
		\scalebox{1}{\includegraphics[width=\linewidth]{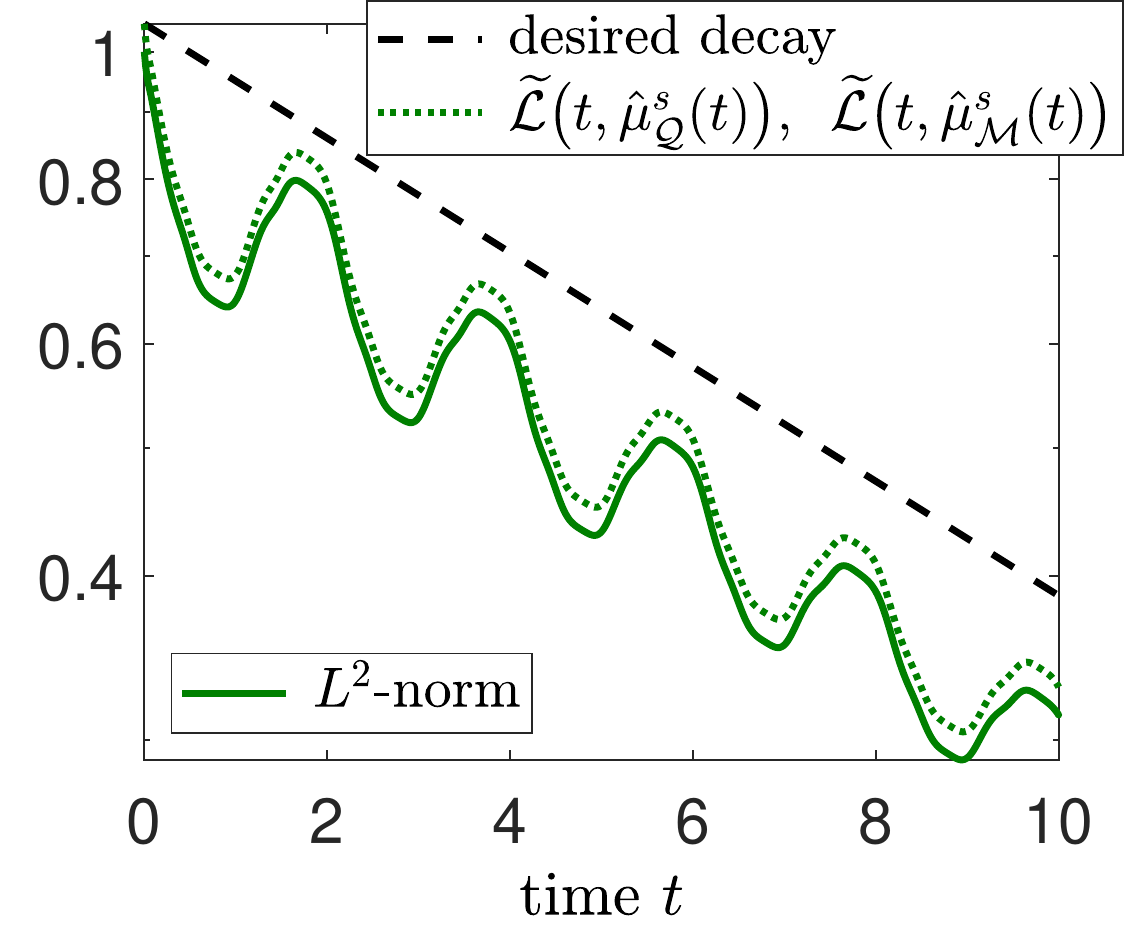}}
	\end{minipage}
	\hfil
	\begin{minipage}{0.45\textwidth}
		\begin{center}	
			$\boldsymbol{\mu=1}$	
		\end{center}
		\scalebox{1}{\includegraphics[width=\linewidth]{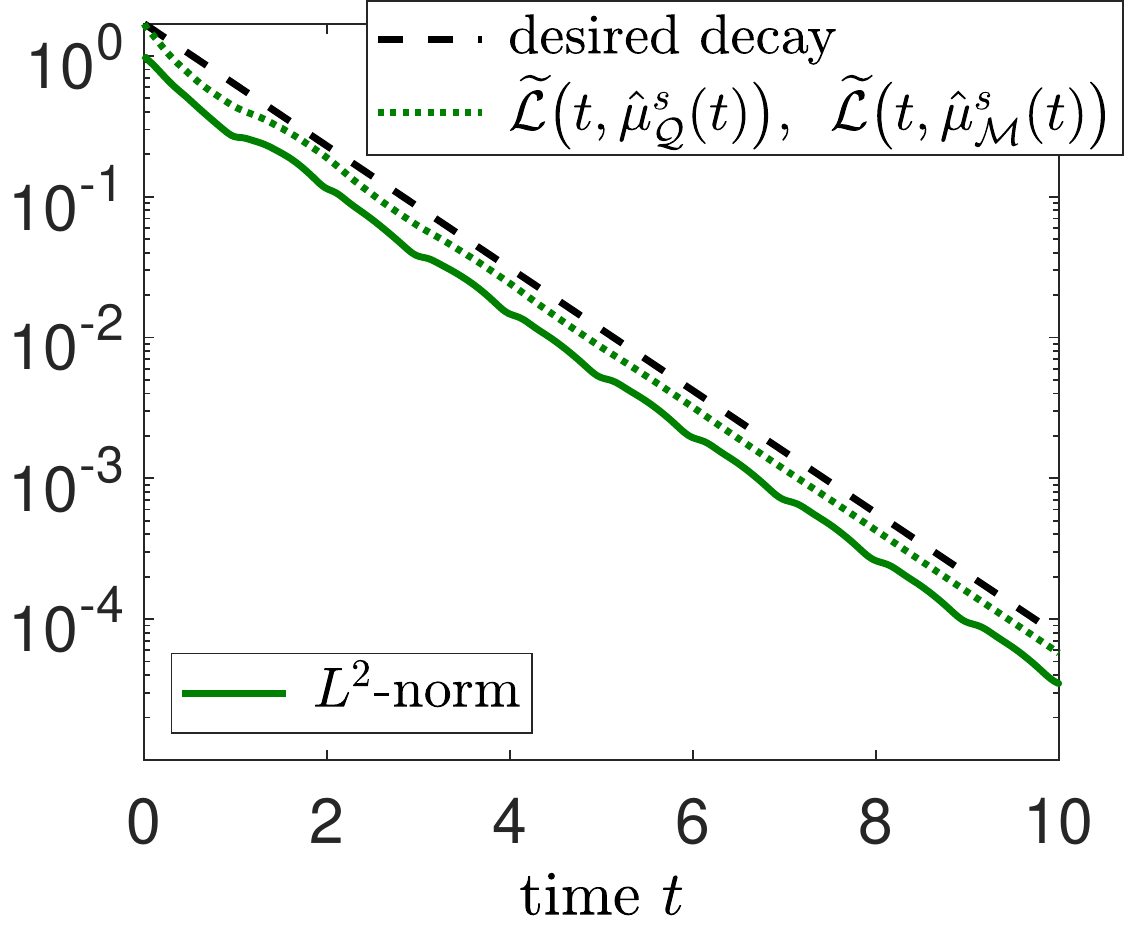}}
	\end{minipage}
	
	\caption{Deviations from steady state are stated in terms of the 
		scaled Lyapunov function~\eqref{LFtilde} as dashed line, which is an upper bound of the $L^2$-norm. Here, simulations for both parameters differ in a magnitude less than~$\sim10^{-5}$.}
	\label{Corollary_3_LF}	
	
\end{figure}


\begin{figure}[H]
	\begin{minipage}{0.45\textwidth}
		\begin{center}	
			$\boldsymbol{\mu=0.1}$	
		\end{center}	
		\scalebox{1}{\includegraphics[width=\linewidth]{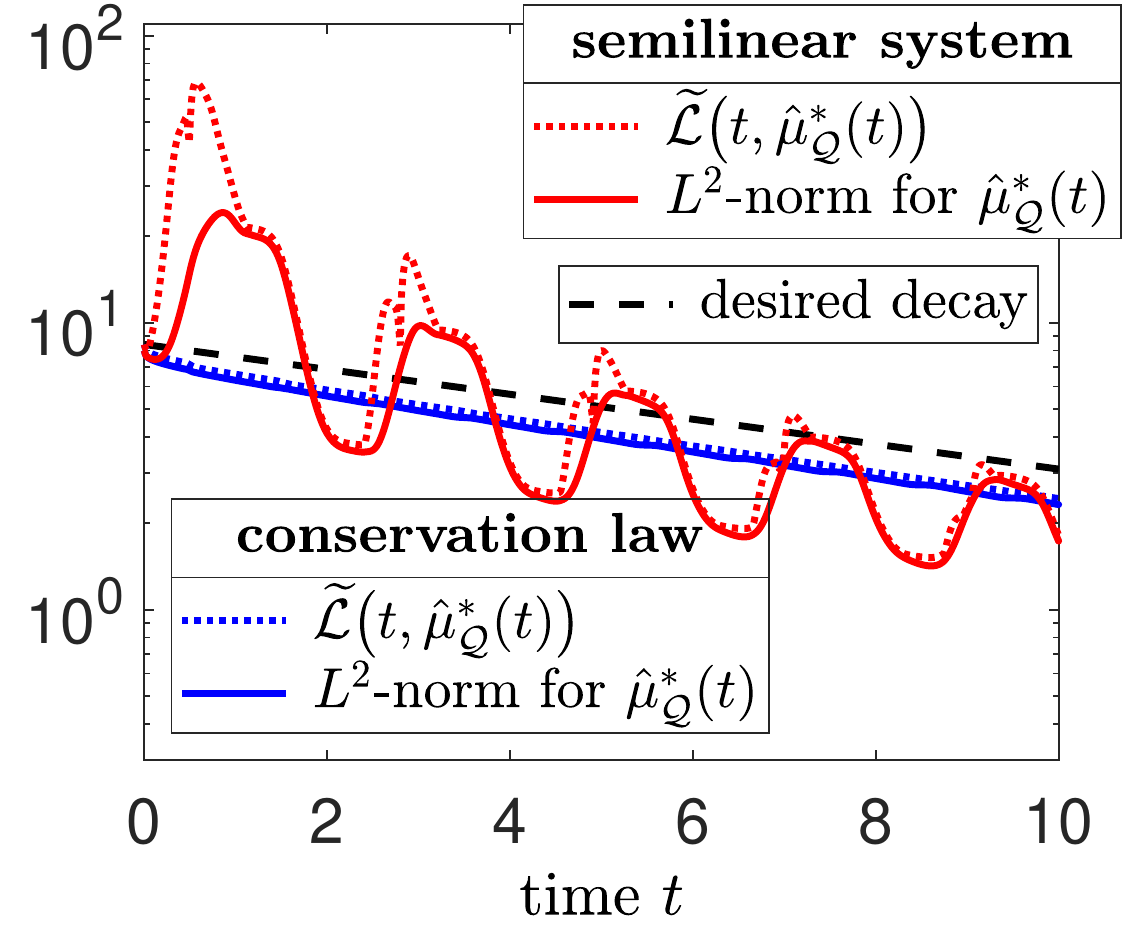}}
	\end{minipage}
	\hfil
	\begin{minipage}{0.45\textwidth}
		\begin{center}	
			$\boldsymbol{\mu=1}$	
		\end{center}
		\scalebox{1}{\includegraphics[width=\linewidth]{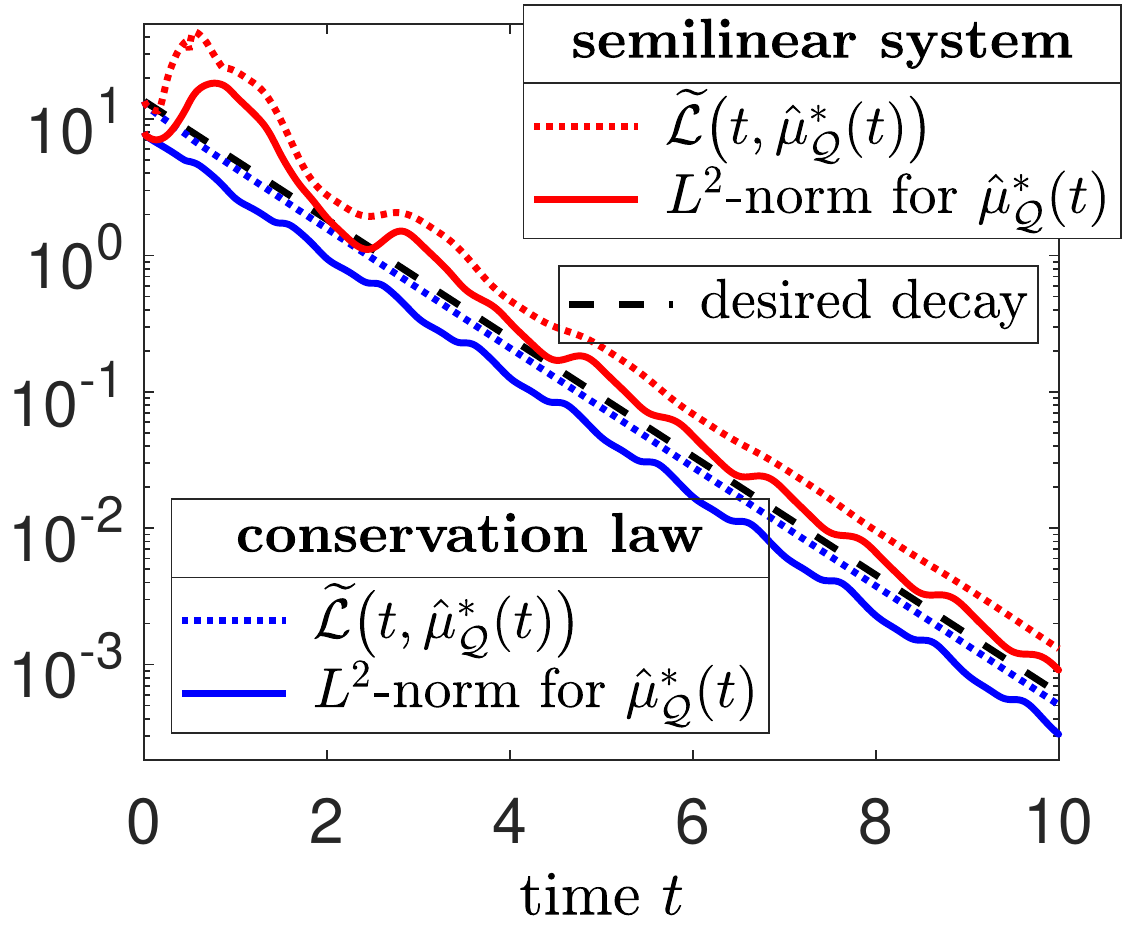}}
	\end{minipage}
	
	\caption{Discontinuous initial data~$
		\Riemann^\pm(0,x) 
		=
		2\,
		\textup{sign}(x)
		$ and feedback control based on the weighted Rayleigh quotient. 
		Simulations for the conservation law without source term are shown in blue. The scaled Lyapunov function~\eqref{LFtilde} yields an upper bound on the $L^2$-norm that  decays at least with the desired rate~$\mu=0.1$ (left panel) and $\mu=1$ (right panel). 
		Simulations for the semilinear system are shown in red. The corresponding deviations decrease non-monotonically.}
	\label{Shock}	
	
\end{figure}

\subsection{Conclusions from numerical experiments}
The numerical results show that the computational framework is able to stabilize semilinear hyperbolic balance laws. 
The control that is based on the weighted Rayleigh quotient is computationally less expensive and leads to a larger control parameter. However, the obtained control may be not sufficient to establish exponential decay in general. 

In contrast, 
the control that uses the parameter~$\mud_{{\M}(t)}$ is computationally more expensive and yields a smaller control parameter. However, this choice can ensure  exponential decay to a desired state for problems where the control based on the weighted Rayleigh quotient only counteracts instabilities. 

Finally, we recall that the numerical Lyapunov function is not necessarily strictly decreasing, since the involved parameters are time-dependent. Still, the presented framework is consistent with theoretical results when a fixed parameter~$\mud(t)=\mub$ is used.

\section*{Summary}
	We have considered the numerical treatment of stabilization problems for semilinear hyperbolic boundary value problems. A Lyapunov function that yields an upper bound on the~$L^2$-norm of the distance to a desired state
is used as an analytical  tool. It is defined up to a parameter that is calculated numerically by a high-order CWENO reconstruction.
The computational framework gives sufficient conditions on a feedback control  to steer the system to a desired state. Numerical experiments illustrate the applicability of the presented approach and its consistency to previous theoretical findings.

\section*{Acknowledgments}

This work is supported by the PRIME programme of the German Academic Exchange Service (DAAD).  
The authors acknowledge support
from ``National Group for Scientific Computation (GNCS-INDAM)'' and by MUR
(Ministry of University and Research) PRIN2017 project number 2017KKJP4X.\\

\noindent
Furthermore, we would like to offer special thanks to the anonymous reviewers for their  valueable feedback.

\section*{Conflict of interest}

The authors declare that they have no conflict of interest.


\bibliographystyle{AIMS}
\bibliography{mybib}

\medskip

\end{document}